\documentclass[11pt]{amsart}
\usepackage[centertags]{amsmath}
\usepackage{amsfonts,amssymb}
\usepackage{graphicx}
\usepackage{enumerate}
 \usepackage{url}

  \newtheorem{theorem}{Theorem}
  \newtheorem{corollary}{Corollary}

  \newtheorem{proposition}{Proposition}
  \newtheorem{lemma}{Lemma}
\newtheorem{question}{Question}
  
  \DeclareMathOperator{\re}{Re}

\DeclareMathOperator{\res}{Res}

\begin{document}

\title[Multiple dense divisibility]{An extension of smooth numbers: \\ multiple dense divisibility}

\author{Garo Sarajian and Andreas Weingartner} 
\address{Department of Mathematical Sciences, United States Military Academy, 601 Cullum Road, West Point, New York, 10996, USA}
\email{gsarajian@gmail.com}
\address{Department of Mathematics, Southern Utah University, 351 West University Boulevard, Cedar City, Utah 84720, USA} 
\email{weingartner@suu.edu} 

\maketitle

\begin{abstract}
The $i$-tuply $y$-densely divisible numbers were introduced by a Polymath project, as a weaker condition on the moduli than $y$-smoothness, 
in distribution estimates for primes in arithmetic progressions. 
We obtain the order of magnitude of the count of these integers up to $x$, uniformly in $x$ and $y$, for every fixed natural number $i$.
\end{abstract}

\section{Introduction}

Let $\mathcal{S}_y$ denote the set of natural numbers whose prime factors are all less than or equal to $y$. 
Such integers are called $y$-smooth (or $y$-friable) and they appear in many different areas of number theory \cite{Gran, HilTen, Pom}. 
Zhang \cite{Zhang} established the existence of infinitely many pairs of primes with bounded gaps 
via distribution estimates for primes in arithmetic progressions to smooth moduli.
In the Polymath paper \cite{Poly}, the condition of $y$-smoothness for the moduli 
is relaxed to the condition of being $i$-tuply $y$-densely divisible, allowing for a larger set of moduli in their estimates. 
The intent is to have a set of integers, as large as possible, where we have control over the location of enough ``nice" divisors, 
i.e. divisors who themselves have enough (nice) divisors whose location we can control, etc. 
In Theorem \ref{thmDkasymp}, we determine the order of magnitude of the count of $i$-tuply $y$-densely divisible numbers up to $x$, uniformly in $x$ and $y$, for every fixed $i \in \mathbb{N}$.

\subsection{The $i$-tuply $y$-densely divisible numbers}

We say that a finite, increasing sequence of natural numbers is \emph{$y$-dense} if the ratios of consecutive terms of the sequence are all $\le y$. 

Define $\mathcal{D}_{0,y}=\mathbb{N}$. For each integer $i\ge 1$, let $n\in \mathcal{D}_{i,y}$ if and only if the increasing sequence $\{d: d|n, d\in \mathcal{D}_{i-1,y}\}$ forms a $y$-dense sequence containing $1$ and $n$. We say that the members of  $\mathcal{D}_{i,y}$ are $i$-tuply $y$-densely divisible, or $(i,y)$-densely divisible, which we sometimes abbreviate as $(i,y)$-d.d.

For example, when $y=2$, the members up to $32$ are given by
$$
\mathcal{D}_{1,2}=\{ 1, 2, 4, 6, 8, 12, 16, 18, 20, 24, 28, 30, 32, ... \},
$$
$$
\mathcal{D}_{2,2}=\{ 1, 2, 4, 8, 12, 16, 24, 32,...\},
$$
$$
\mathcal{D}_{3,2}=\{ 1, 2, 4, 8, 16, 24, 32,...\},
$$
and, for $i\ge 4$,
$$
\mathcal{D}_{i,2}=\{ 1, 2, 4, 8, 16, 32,...\}.
$$

In general, the definition of  $\mathcal{D}_{i,y}$ implies 
$$\mathcal{D}_{i+1,y}\subseteq \mathcal{D}_{i,y}\qquad (i\ge 0).$$ 
By Corollary \ref{corS},
$$
\mathcal{S}_y = \bigcap_{i\ge 0} \mathcal{D}_{i,y}.
$$
Thus, the integers that are $i$-tuply $y$-densely divisible for all $i\ge 0$ are exactly the $y$-smooth numbers.

The following somewhat stronger definition is used in Polymath \cite{Poly}:

Define $\mathcal{D}^*_{0,y}=\mathbb{N}$. For $i\ge 1$, let $n\in \mathcal{D}^*_{i,y}$ if and only if for each $j$, $0\le j \le i-1$, 
the increasing sequence $\{d: d|n, d\in \mathcal{D}^*_{j,y} \text{ and } n/d \in \mathcal{D}^*_{i-1-j,y} \}$ forms a $y$-dense sequence containing $1$ and $n$. 
We call the members of  $\mathcal{D}^*_{i,y}$ strongly $i$-tuply $y$-densely divisible, or strongly $(i,y)$-densely divisible.

As above, the definition of  $\mathcal{D}^*_{i,y}$ implies
$$\mathcal{D}^*_{i+1,y}\subseteq \mathcal{D}^*_{i,y} \qquad (i\ge 0),$$
and, by Corollary \ref{corS},
$$
\mathcal{S}_y = \bigcap_{i\ge 0} \mathcal{D}^*_{i,y}.
$$

Considering $j=i-1$ in the definition of $\mathcal{D}^*_{i,y}$, we see that $\mathcal{D}^*_{i,y}\subseteq \mathcal{D}_{i,y}$ for all $i\ge 0$. 
Note that $\mathcal{D}_{i,y}=\mathcal{D}^*_{i,y}$ for $0\le i \le 2$. 
The members up to $32$ displayed above are unchanged if $\mathcal{D}_{i,2}$ is replaced by $\mathcal{D}^*_{i,2}$. 
However, $8424 \in \mathcal{D}_{3,2}\setminus \mathcal{D}^*_{3,2}$ and $65520 \in \mathcal{D}_{4,2}\setminus \mathcal{D}^*_{4,2}$.

We define the corresponding counting functions as
$$ 
D_i(x,y): = | [1,x] \cap \mathcal{D}_{i,y} |, \quad D^*_i(x,y): = | [1,x] \cap \mathcal{D}^*_{i,y} |.
$$

\begin{theorem}\label{thmDkasymp}
Let $i \in \mathbb{N}$ be fixed. Uniformly for $x\ge y \ge 2$,
$$
 D_i(x,y) \asymp_i D_i^*(x,y)   \asymp_i x \left(\frac{\log y }{\log x}\right)^{\lambda_{1/i}},
$$
where the exponent $\lambda_{1/i}$ is defined in Theorem \ref{rhoathm} and its numerical values are listed in Table \ref{tablelambda}.
These estimates remain valid if only squarefree integers in $\mathcal{D}_{i,y}$ (resp. $\mathcal{D}^*_{i,y}$)  are counted, provided $x\ge y \ge y_0(i)$. 
As $i\to \infty$, we have 
$$\lambda_{1/i}= i \bigl(\log (i) -1 +o(1)\bigr).$$
\end{theorem}

\begin{table}[h]
  \begin{tabular}{ | l | l | }
   \hline
    $i$ & $\lambda_{1/i}$ \\ \hline
    1 & 1  \\ \hline
2 & 2.46206...  \\ \hline
3 & 4.20605... \\ \hline
4 & 6.15900... \\ \hline
5 & 8.27925...  \\ \hline
  \end{tabular}
\quad 
 \begin{tabular}{ | l | l | }
   \hline
    $i$ & $\lambda_{1/i}$ \\ \hline
6 & 10.5395... \\ \hline
7 & 12.9203... \\ \hline
8 & 15.4074...  \\ \hline
9 & 17.9892...   \\ \hline
10 & 20.6568...  \\ \hline
    \end{tabular}
\quad 
 \begin{tabular}{ | l | l | }
   \hline
    $i$ & $\lambda_{1/i}$ \\ \hline
11 & 23.4026... \\ \hline
12 & 26.2206... \\ \hline
13 & 29.1054...  \\ \hline
14 & 32.0524...   \\ \hline
15 & 35.0578...  \\ \hline
    \end{tabular}
\quad 
 \begin{tabular}{ | l | l | }
   \hline
    $i$ & $\lambda_{1/i}$ \\ \hline
16 & 38.1180... \\ \hline
17 & 41.2300... \\ \hline
18 & 44.3909...  \\ \hline
19 & 47.5985...   \\ \hline
20 & 50.8504...  \\ \hline
    \end{tabular}
  \caption{Truncated values of the exponent in Theorem \ref{thmDkasymp}, obtained from the estimate \eqref{gasapprox}.}
  \label{tablelambda}
  \end{table}

Theorem \ref{thmDkasymp} follows from Theorems \ref{thm5sets}, \ref{rhoathm}, \ref{thmBaxyIntro} and \ref{thmlambda}.

When $i=1$, we have $D_1(x,y)=D^*_1(x,y) \asymp x \frac{\log y}{\log x}$, a result first established by Saias \cite[Thm. 1]{Saias1},
which has since been improved to an asymptotic estimate \cite{PDD} (see also Corollary \ref{cor1Baxy} with $a=1$).
With Table \ref{tablelambda}, we obtain the following new estimates, valid uniformly for $x\ge y\ge 2$:
$$
D_2(x,y)=D^*_2(x,y)\asymp x  \left(\frac{\log y }{\log x}\right)^{2.46206...},
$$
$$
D_3(x,y)\asymp D^*_3(x,y)\asymp x  \left(\frac{\log y }{\log x}\right)^{4.20605...},
$$
$$
D_4(x,y)\asymp D^*_4(x,y)\asymp x  \left(\frac{\log y }{\log x}\right)^{6.15900...},
$$
et cetera. If we count squarefree integers in $\mathcal{D}_{i,y}$ (resp. $\mathcal{D}^*_{i,y}$), 
these same estimates hold uniformly for $x\ge y \ge y_0(i)$. When $i=1$, Saias \cite[Thm. 1]{Saias1} showed that we can take $y_0(1)=2$. 
For general $i\ge 1$, \cite[Lemma 2.2, Part 4]{Sar} implies that $y_0(i)$ cannot be less than the $i$-th prime number.

\subsection{Describing $ \mathcal{D}_{i,y}$ (resp. $ \mathcal{D}^*_{i,y}$) by conditions on the prime factors}\label{subsecBtheta}

Let $\theta$ be a real arithmetic function and let $\mathcal{B}_\theta$ be the set of positive integers containing $n=1$ and all $n=p_1\cdots p_k$, 
whose prime factors $p_1\le \cdots \le p_k$ satisfy $p_{j+1} \le \theta ( p_1\cdots p_{j})$  for $0\le j < k$. 
Tenenbaum \cite{Ten86} showed that  $\mathcal{D}_{1,y} =\mathcal{B}_\theta$ with $\theta(n)=yn$.

\begin{lemma}[Tenenbaum \cite{Ten86}]\label{lemTen}
Let $y\ge 2$ and $n=p_1 \cdots p_k$, where $p_1\le \cdots \le p_k$ are prime. Then $n \in \mathcal{D}_{1,y}$ if and only if 
\begin{equation*}\label{pcond}
p_{j+1} \le y  p_1 \cdots p_j \qquad (0\le j<k).
\end{equation*}
That is, $ \mathcal{D}_{1,y} =\mathcal{D}_{1,y}^* = \mathcal{B}_\theta$ with $\theta(n)=yn$. 
\end{lemma}

Proposition \ref{thmi2} reveals $\theta(n)$ when $i=2$ and shows that  for all $i\ge 1$, we have $ \mathcal{D}_{i,y} = \mathcal{B}_{\theta_{i,y}}$ for some function $\theta_{i,y}$.
The case $i=2$ suggests that there is no simple explicit formula for  $\theta_{i,y}$ when $i\ge 2$.
We state the following result without proof, since we will not be using it here.
\begin{proposition}\label{thmi2}
We have
 $ \mathcal{D}_{2,y} =\mathcal{D}^*_{2,y} = \mathcal{B}_\theta$ with $\theta(n) = y \max\limits_{d \in \mathcal{D}_{1,y} \atop d| n} \min(n/d,d)$.
For $i\ge 1$,  $ \mathcal{D}_{i,y} = \mathcal{B}_{\theta_{i,y}}$ with $\theta_{i,y}(n) =\max(\{1\}\cup \{p: p \text{ prime, } np \in\mathcal{D}_{i,y}\})$.
\end{proposition}

\begin{question}
Let $i\ge 3$. Is $\mathcal{D}^*_{i,y}=\mathcal{B}_{\theta^*_{i,y}}$ for some function $\theta^*_{i,y}$? 
\end{question}

Even without an answer to this question, we can determine the order of magnitude of the counting functions of the sets 
$ \mathcal{D}_{i,y}$ and $ \mathcal{D}^*_{i,y}$, with the help of the following theorem.
\begin{theorem}\label{thm5sets}
For all $i \in \mathbb{N}$ and real $y\ge 2$,
$$
\mathcal{S}_y\subseteq \mathcal{B}_{\theta_l} \subseteq \mathcal{D}^*_{i,y}  \subseteq \mathcal{D}_{i,y}  \subseteq \mathcal{B}_{\theta_u},
$$
where 
$$
\theta_l(n) = \max(y, (yn)^{1/i}), \quad \theta_u(n) = y n^{1/i}.
$$
\end{theorem}

The proof of Theorem \ref{thm5sets} is given in Section \ref{sec5sets}.

The sets $\mathcal{B}_{\theta_l}$ and $\mathcal{B}_{\theta_u}$ are easier to understand than the sets $ \mathcal{D}_{i,y} $ and $ \mathcal{D}^*_{i,y} $,
since the former two are defined by explicit conditions on the prime factors, whereas the latter two are defined recursively with conditions on the divisors.
Theorem \ref{thmBaxyIntro} provides estimates for the counting functions of $\mathcal{B}_{\theta_l}$ and $\mathcal{B}_{\theta_u}$.
Theorem \ref{thm5sets} then implies the desired estimates  in Theorem \ref{thmDkasymp}.

\begin{corollary}\label{corS}
For $y\ge 2$ we have
$$
\mathcal{S}_y =  \bigcap_{i\ge 0} \mathcal{D}_{i,y}= \bigcap_{i\ge 0} \mathcal{D}^*_{i,y} .
$$
\end{corollary}

\begin{proof}[Proof of Corollary \ref{corS}]
The first three inclusions of 
$$
\mathcal{S}_y \subseteq  \bigcap_{i\ge 0} \mathcal{D}^*_{i,y}  \subseteq  \bigcap_{i\ge 0} \mathcal{D}_{i,y}  \subseteq  \bigcap_{i\ge 0} \mathcal{B}_{\theta_u}   \subseteq \mathcal{S}_y
$$
follow at once from Theorem \ref{thm5sets}.
To show the last inclusion, first note that $1$ is in each set. Assume $m\ge 2$ satisfies $m \in \bigcap_{i\ge 0} \mathcal{B}_{\theta_u}$.
Let $p$ be the largest prime factor of $m$ and write $m=np$.
Then $m \in   \mathcal{B}_{\theta_u}$ implies $p \le \theta_u(n) = y n^{1/i}$. Since this is true for all $i\ge 0$, it follows that $p\le y$ and $m \in \mathcal{S}_y$. 
\end{proof}

\subsection{An extension of Dickman's function}

The counting function of the set $\mathcal{S}_y$ can be approximated in terms of Dickman's function $\rho(u)$. 
For the other four sets appearing in Theorem \ref{thm5sets}, the counting functions can be estimated in terms of $\rho_{1/i}(u)$, which we define as follows
(see \cite[Thm. 2.9]{Sar}):

For $a\ge 0$, define $\rho_a(u)$ by $\rho_a(u)=0$ for $u<0$ and 
\begin{equation}\label{rhoadef}
\rho_a(u) = 1-\int_0^{\frac{u-1}{1+a}}  \rho_a(v)\, \omega\left(\frac{u-v}{1+av}\right) \frac{dv}{1+av} \qquad (u\ge 0),
\end{equation}
where $\omega(u)$ denotes Buchstab's function.
When $0\le u\le 1$, the integral vanishes and $\rho_a(u)=1$. 
Note that  \eqref{rhoadef} expresses $\rho_a(u)$ based on values of 
$\rho_a(v)$ with $v\le \frac{u-1}{1+a} \le u-1$. 
Thus \eqref{rhoadef} defines $\rho_a(u)$ uniquely for all $u\ge 0$ and provides a method for calculating $\rho_a(u)$ numerically. 
When $a=0$, \eqref{rhoadef} reduces to the convolution identity \cite[Exercise 297]{Ten} for Dickman's function $\rho(u)$, which shows that 
$\rho_0(u)=\rho(u)$ for all $u$. 
When $a=1$, \eqref{rhoadef} coincides with the definition of the function $d(u)$ in \cite{IDD3, PDD},
which plays a role in the distribution of $y$-densely divisible numbers (i.e. $(1,y)$-d.d.) and satisfies
$$
\rho_1(u)=d(u)=\frac{C}{u+1}\left(1+O(u^{-2.03})\right), \qquad C=\frac{1}{1-e^{-\gamma}}=2.280...
$$

\begin{theorem}\label{rhoathm}
Let $a>0$ be fixed. For $u\ge 0$, $\rho_a(u)$ is decreasing in $u$ and 
$$
\rho(u) \le \rho_a(u) \le 1.
$$
Let $-\lambda_a$ be the real zero of \eqref{eqgaseval} with maximal real part.
For the  values of $a$ and $C_a$ in Table \ref{table0}, we have
\begin{equation}\label{rhoasymp0}
\rho_a(u) = \frac{C_a}{(1+au)^{\lambda_a}}\left(1+O_a\left(\frac{1}{u^{2.03}}\right)\right) \qquad (u\ge 1).
\end{equation}
For all $a\ge 1$,  we have
\begin{equation}\label{rhoasymp1}
\rho_a(u) = \frac{C_a}{(1+au)^{\lambda_a}}\left(1+O_a\left(\frac{1}{u^{1+\varepsilon_a}}\right)\right) \qquad (u\ge 1),
\end{equation}
where $C_a$ is a positive constant and $\varepsilon_a >0$. 
For all $a>0$, we have
\begin{equation}\label{rhoasymp2}
\rho_a(u) \asymp_a  \frac{1}{(1+au)^{\lambda_a}}\qquad (u\ge 0).
\end{equation}
\end{theorem}

 \begin{figure}[h]
\begin{center}
\includegraphics[height=6.531cm,width=12cm]{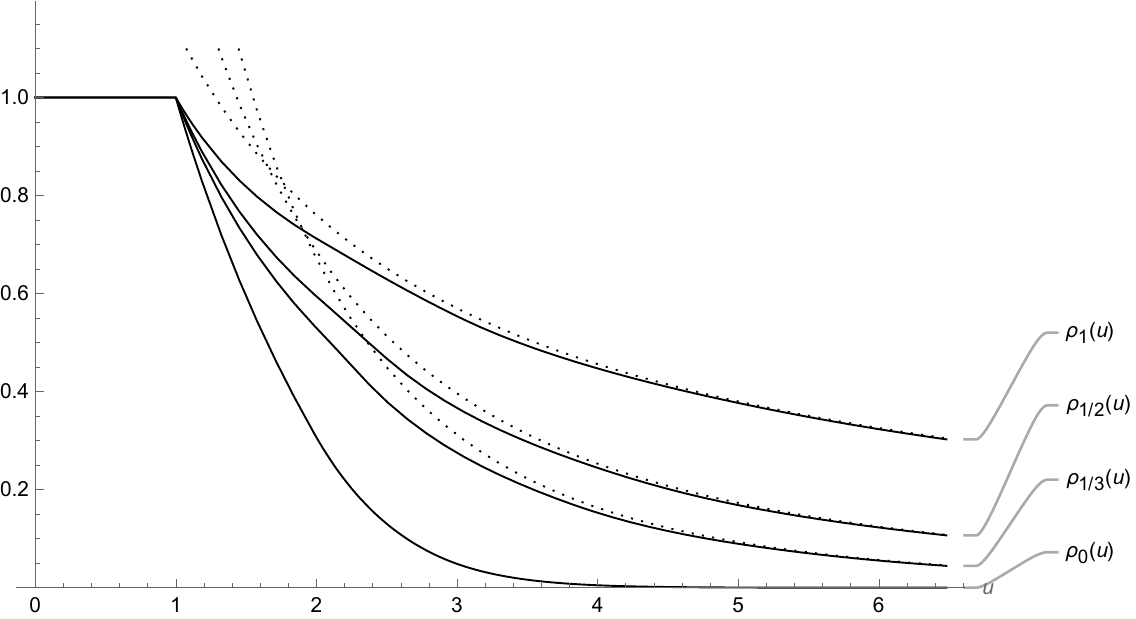}
\caption{Graphs of $\rho_a(u)$ (solid), for $a=1,\frac{1}{2}, \frac{1}{3},0$, and of the corresponding asymptotes $\frac{C_a}{(1+au)^{\lambda_a}}$ (dotted), when $a\neq 0$.}
\label{figure1}
\end{center}
\end{figure}

\begin{table}[h]
  \begin{tabular}{ | l | l | l | l | l |}
   \hline
    $a$ & $\lambda_a$ & $C_a$\\ \hline
    1 & 1  & $\frac{1}{1-e^{-\gamma}}$  \\ \hline
1/2 & 2.46206... & 3.7815...  \\ \hline
1/3 & 4.20605... & 5.7645...\\ \hline
1/4 & 6.15900... & 8.3827...\\ \hline
1/5 & 8.27925... & 11.812... \\ \hline
 \end{tabular}
\qquad
  \begin{tabular}{ | l | l | l | l | l |}
   \hline
    $a$ & $\lambda_a$ & $C_a$\\ \hline
1/6 & 10.5395... & 16.265... \\ \hline
1/7 & 12.9203... &  22.000... \\ \hline
1/8 & 15.4074... &  29.333... \\ \hline
1/9 & 17.9892... &  38.648...  \\ \hline
1/10 & 20.6568... & 50.410...  \\ \hline
    \end{tabular}
  \caption{Values of $\lambda_a$,  $C_a$.}
  \label{table0}
  \end{table}

We conjecture that \eqref{rhoasymp0} holds for all $0<a \le 1$. The proof of Theorem \ref{rhoathm} is given in Section \ref{secrhoa}.

\subsection{Estimates for the counting functions}

Let $\mathcal{B}_{a,y}$ be the set corresponding to $\theta(n) = y n^{a}$ and let $B_{a}(x,y)$ denote the corresponding counting function.
That is, $B_a(x,y)$ equals
$$
1+\sum_{k\ge 1}|\{n=p_1\cdots p_k \le x, \, p_1\le \cdots \le p_k, \, p_{j+1}\le y(p_1\cdots p_j)^a, \, 0\le j <k\}|.
$$
Similarly, let $\mathcal{B}^*_{a,y}$ be the set corresponding to $\theta(n) = \max(y, (yn)^{a})$ and 
let $B^*_{a}(x,y)$ denote the corresponding counting function.
Let $\tilde{\mathcal{B}}_{a,y}$ (resp.  $\tilde{\mathcal{B}}^*_{a,y}$) denote the set of squarefree members of $\mathcal{B}_{a,y}$ (resp. $\mathcal{B}^*_{a,y}$), and let $\tilde{B}_a(x,y)$ (resp. $\tilde{B}^*_a(x,y)$) be the corresponding counting functions.
Let 
$$
u=\frac{\log x}{\log y}.
$$

\begin{theorem}\label{thmBaxyIntro}
 Let $a>0$ be fixed. Uniformly for $x\ge y\ge 2$, we have
$$
B_a(x,y) = x \,\rho_a(u) \left(1+O_a\left(\frac{1}{\log y} \right)\right),
$$
$$
\tilde{B}_a(x,y) = \frac{x \,\rho_a(u)}{\zeta(2)} \left(1+O_a\left(\frac{1}{\log y} \right)\right),
$$
where $\zeta(2)=\pi^2/6$, and
$$
 B_a(x,y) \asymp_a B^*_{a}(x,y) \asymp_a x \rho_a(u) .
$$

Uniformly for $x\ge y\ge y_0(a)$, we have
$$
 \tilde{B}_a(x,y) \asymp_a \tilde{B}^*_{a}(x,y) \asymp_a x \rho_a(u) .
$$

\end{theorem}

\begin{proof}
This follows from Corollaries \ref{cor1Baxy} and \ref{corBaxy}.
\end{proof}

\subsection{An extension of the Schinzel-Szekeres function}

As in \cite{Sar}, for the real parameter $\beta$ we define $F_\beta(1)=1$ and
$$
F_\beta(n) = \max_{d|n \atop d>1} d (P^-(d))^\beta \qquad (n\ge 2),
$$
where $P^-(d)$ denotes the smallest prime factor of $d>1$.
When $\beta=1$, this is the Schinzel-Szekeres function \cite{Ten86,SSF}. 
Note that we have
\begin{equation}\label{eqSSFcount}
 |\{n\le x: F_\beta(n)/n \le y^\beta \}|=B_{1/\beta}(x,y) ,
\end{equation}
which we estimated in Theorem \ref{thmBaxyIntro}.
For $x\ge 1$, $y\ge 1$, define
$$
\mathcal{A}_\beta(x,y):= \{n\le x : F_\beta(n)\le xy\}, \quad \tilde{\mathcal{A}}_\beta(x,y):=\mathcal{A}_\beta(x,y) \cap \{n :  \mu^2(n)=1\},
$$
and $A_\beta(x,y)=|\mathcal{A}_\beta(x,y)|$, $\tilde{A}_\beta(x,y)=|\tilde{\mathcal{A}}_\beta(x,y)|$. 
The integers in $\tilde{\mathcal{A}}_\beta(D,1)$ 
make up the intersection of the support of the upper bound beta sieve and the support of the lower bound beta sieve of level $D$ \cite[Ch. 11]{Opera}.

\begin{theorem}\label{thmAbeta}
Let $\beta>0$. For $x\ge 1$, $y\ge 1$, we have
$$
A_\beta(x,y) \asymp_\beta \tilde{A}_\beta(x,y) \asymp_\beta x \left(\frac{\log 2y}{\log 2xy}\right)^{\lambda_{1/\beta}},
$$
where $\lambda_{1/\beta}$ is as in Theorem \ref{rhoathm} and Table \ref{tablelambda}.
\end{theorem}

When $\beta=1$, we have $A_1(x,y) \asymp x \frac{\log 2y}{\log 2xy}$, which was established by Saias \cite{Saias1}, 
and later improved to an asymptotic estimate \cite{SSF}.

We will derive Theorem \ref{thmAbeta} in Section \ref{secAbeta} from Theorem \ref{thmBaxyIntro}.

\subsection{Asymptotic estimates for the exponent $\lambda_a$}

When $a\ge 1$, the explicit inequalities \cite[Eq. (8), (9)]{SPA} imply
$$
\lambda_a = \frac{e^{-\gamma}}{a}\left(1+\frac{e^{-\gamma} \log (a+1)}{a} +\frac{r_a}{a} \right) \qquad (a\ge 1),
$$
where $|r_a|\le\frac{2}{3}$.
Theorem \ref{thmlambda} reveals the behavior of $\lambda_a$ as $a\to 0^+$.

\begin{theorem}\label{thmlambda}
The exponent $\lambda_a$, defined in Theorem \ref{rhoathm}, satisfies
$$
\lambda_a = \frac{1}{a}\left(\log \left(\frac{1}{a}\right) -1 +o(1)\right)\qquad (a \to 0^+).
$$
\end{theorem}

Thus, $\lambda_{1/i}=i \bigl(\log(i)-1+o(1)\bigr)$ as $i\to \infty$, 
as claimed in Theorem \ref{thmDkasymp}.
The proof of Theorem \ref{thmlambda}, which takes a considerable amount of effort, is given in Section \ref{seclambda}.

\section{Proof of Theorem \ref{thm5sets}}\label{sec5sets}

In this section, we write $\theta_{i,y}(n) := \theta_l(n) = \max(y, (yn)^{1/i})$ and $\mathcal{L}_{i,y}:=\mathcal{B}_{\theta_l}$ for $i\ge 1$. We define $\mathcal{L}_{0,y}:= \mathbb{N}$. 

\begin{lemma}\label{lemprop}
Let $i\ge 1$, $y\ge 2$ and $n \in \mathcal{L}_{i,y}$. Then for all $1\le R\le yn$ and all integers $v, w \ge 0$ with $v+w=i-1$, there is a factorization
$n=d_v d_w$ with $R/y \le d_w \le R$ and $d_w \in \mathcal{L}_{w,y}$, $d_v \in \mathcal{L}_{v,y}$. 
\end{lemma}

\begin{proof}
Let $i\ge 1$ and $y\ge 2$. We prove the lemma by induction on $k=\Omega(n)$. The lemma clearly holds for $n=1$, i.e. $k=0$, since $1=1\cdot 1$. 
When $k=1$, $n=p$, a prime. Then $n\in \mathcal{L}_{i,y}$ means $n=p \le \theta_{i,y}(1)=y$. 
If $n<R\le yn$, we take $d_w=n$ and $d_v=1$. If $1\le R \le n$, take $d_w=1$ and $d_v=n$. 
Since $n=p\le y$, both $1$ and $n$ are members of $\mathcal{L}_{v,y}$ and of $\mathcal{L}_{w,y}$ for all $v,w\ge 0$. 

Let $k\ge 1$ and assume the lemma holds for all $n$ with $\Omega(n)\le k$. 
Now let $\Omega(n)=k+1$, $n=mp$, $p\ge P^+(m)$, $\Omega(m)=k$ and $n\in  \mathcal{L}_{i,y}$.
Then $m\in  \mathcal{L}_{i,y}$ and $p \le \theta_{i,y}(m)=\max(y,(ym)^{1/i})$. 
Let $1\le R \le yn$ and  $v,w \ge 0$ be integers with $v+w=i-1$. We consider two cases:

Case (1): Assume $1\le R \le \frac{my}{p^v}$. By the inductive hypothesis, there is a factorization of $m=d_v d_w$ with $R/y \le d_w \le R$,
$d_w  \in \mathcal{L}_{w,y}$, $d_v \in \mathcal{L}_{v,y}$. Define $d_v'=d_v p$. Then $d_v=m/d_w$ and $d_w\le R \le \frac{my}{p^v}$ implies 
$d_v \ge \frac{p^v}{y}$. Thus $p^v \le y d_v$, which means that $d_v' \in \mathcal{L}_{v,y}$. The desired factorization is $n=d_w d_v'$. 

Case (2): Assume $ \frac{my}{p^v}< R \le ny$. Define $R':=R/p \le my$. If $p\le y$, then $n\in \mathcal{S}_y$ and the conclusion holds. Thus we may assume $p>y$. 
We claim that $R'\ge 1$. 
Indeed, if $R'<1$, then $R<p$, so $p>R>\frac{my}{p^v}$, which implies $p^i \ge p^{v+1} >my$, contradicting $p \le \theta_{i,y}(m)=\max(y,(ym)^{1/i})$.
Thus $1\le R' \le my$. By the inductive hypothesis, we can write $m=d_v d_w$, with $R'/y \le d_w \le R'$,
$d_w  \in \mathcal{L}_{w,y}$, $d_v \in \mathcal{L}_{v,y}$. Let $d_w'=d_w p$, so $R/y \le d_w' \le R$. We have
$$
d_w \ge \frac{R'}{y}= \frac{R}{py} > \frac{m}{p^{v+1}},
$$
since $R>\frac{my}{p^v}$. Thus
$$
\frac{p^w}{d_w} \le \frac{p^{v+w+1}}{m} = \frac{p^i}{m} \le y,
$$
since $p \le \theta_{i,y}(m)=\max(y,(ym)^{1/i})$ and $p>y$. It follows that $p^w \le y d_w$, which means that $d_w'=p d_w \in \mathcal{L}_{w,y}$. 
The desired factorization is $n=d_w' d_v$. 
\end{proof}

\begin{proof}[Proof of Theorem \ref{thm5sets}]

(i) For the first inclusion, note that $\mathcal{S}_y = \mathcal{B}_\theta$ with $\theta(n)=y$ for all $n\ge 1$. 
Thus the claim $\mathcal{S}_y\subseteq \mathcal{B}_{\theta_l}$
is a consequence of $\theta(n) = y \le  \max(y, (yn)^{1/i})=\theta_l(n)$ for all $n\ge 1$. 

(ii) We show the second inclusion by induction on $i$. 
The case $i=1$ follows from Lemma \ref{lemTen}. 
Let $i\ge 2$ and assume that  $\mathcal{L}_{j,y} \subseteq  \mathcal{D}^*_{j,y} $ for $0\le j \le i-1$. 
Let $n\in \mathcal{L}_{i,y} $ and let $v,w \ge 0$ be integers with $v+w=i-1$. Let $1\le R \le yn$.  
By Lemma \ref{lemprop}, there is a factorization
$n=d_v d_w$ with $R/y \le d_w \le R$ and $d_w \in \mathcal{L}_{w,y}$, $d_v \in \mathcal{L}_{v,y}$.
Since $0\le v, w \le i-1$, the inductive hypothesis ensures that $d_w \in \mathcal{D}^*_{w,y}$, $d_v \in \mathcal{D}^*_{v,y}$. Thus $n\in \mathcal{D}^*_{i,y}$.

(iii) The third inclusion is obvious from the definition of $\mathcal{D}_{i,y}$ and $ \mathcal{D}^*_{i,y}$. 

(iv) We show the last inclusion by induction on $i$. 
The case $i=1$ follows from Lemma \ref{lemTen}. 
 Assume that $n\in \mathcal{D}_{i,y}$ for some $i\ge 2$ and let $p$ be an arbitrary prime factor of $n$. 
Write $n=m p r$ with $P^+(m) \le p \le P^-(r)$, where $P^+(k)$ (resp. $P^-(k)$) is the largest (resp. smallest) prime factor of $k\ge 2$, $P^+(1)=1$ and $P^-(1)=\infty$.
Let $ds$ be the smallest divisor of $n$ such that $ds \in \mathcal{D}_{i-1,y}$, $d|m$, $s|pr$ and $s>1$.
Since $n\in \mathcal{D}_{i-1,y}$, there is at least one such divisor. 
By the inductive hypothesis, $p\le P^-(s) \le y d^{1/(i-1)}$ so that $dp \ge p^i/y^{i-1}$. Now $pd \le sd \le ym$, since the sequence of divisors of 
$n$, that are in $\mathcal{D}_{i-1,y}$, is $y$-dense. It follows that
$ym \ge  p^i/y^{i-1}$, that is $p \le y m^{1/i}$. Thus $n\in \mathcal{B}_{\theta_u}$.
\end{proof}

\section{Proof of Theorem \ref{rhoathm}}\label{secrhoa}

\subsection{Proof that $\rho_a(u)$ is decreasing in $u$ and that $\rho(u)\le \rho_a(u)\le 1$}

Let $\mathcal{B}_{a,y}$ be the set corresponding to $\theta(n) = y n^{a}$ and let $B_{a}(x,y)$ denote the corresponding counting function.
The following preliminary estimate of $B_a(x,y)$, which assumes $u$ is fixed, is established in \cite[Thm. 2.9]{Sar} by induction on $\lfloor u \rfloor$.

\begin{proposition}
Let $a>0$ and $u\ge 1$ both be fixed and let $y=x^{1/u}$. We have
$$
B_{a}(x,y) \sim  x \rho_a(u) \qquad (x\to \infty).
$$
\end{proposition}

Let $\Psi(x,y)$ denote the number of $y$-smooth integers up to $x$. With the conditions as in the proposition and $x\to \infty$,
$$
x \rho(u) \sim \Psi(x,y) \le B_{a}(x,y) \sim x \rho_a(u) \sim  B_{a}(x,y) \le x,
$$
which implies $\rho(u)\le  \rho_a(u) \le 1$. 
If $u_1\le u_2$, then 
$$
x \rho_a(u_2) \sim  B_{a}(x,x^{1/u_2}) \le B_{a}(x,x^{1/u_1}) \sim x \rho_a(u_1),
$$
which shows that $\rho_a(u)$ is decreasing in $u$.

\subsection{Asymptotic estimates for $\rho_a(u)$ via Laplace transforms.}

We will often make implicit use of the following estimates for Buchstab's function, where $\Gamma(u)$ denotes the gamma function.
\begin{lemma}\label{omegaest}
We have
\begin{enumerate}[(i)]
\item $|\omega'(u)|\le 1/\Gamma(u+1) \quad (u\ge 0)$,
\item $|\omega(u)-e^{-\gamma}| \le 1/\Gamma(u+1) \quad (u\ge 0)$.
\end{enumerate}
\end{lemma}

\begin{proof} See \cite[Lemma 2.1]{PDD}.
\end{proof}

With the change of variables 
$$e^z=1+au,  \qquad e^t=1+av,$$ 
$$F_a(z):=\rho_a((e^z-1)/a)=\rho_a(u), \qquad \Omega_a(z):=\omega((e^z-1)/a) =\omega(u),$$ 
equation \eqref{rhoadef} (with the upper limit in the integral replaced by $u$) becomes
 \begin{equation}\label{FaDef}
 F_a(z) =  1 - \frac{1}{a} \int_{0}^{z} F_a(t)  \, \Omega_a(z-t) \, dt \qquad (z\ge 0).
 \end{equation}
Taking Laplace transforms, we obtain, for $\re(s)>0$, 
$$
\widehat{F}_a(s) = \frac{1}{s} - \frac{1}{a} \widehat{F}_a(s)\,  \widehat{\Omega}_a(s).
$$
Solving for $\widehat{F}_a(s) $, we get
\begin{equation}\label{FaLap}
\widehat{F}_a(s) =  \frac{1}{s(1+\frac{1}{a} \widehat{\Omega}_a(s)) } .
\end{equation}
The poles of $\widehat{F}_a(s)$ are the zeros of
$$
g_a(s) := \frac{1}{\widehat{F}_a(s)} =  s(1+\frac{1}{a} \widehat{\Omega}_a(s)) =s+s\int_0^\infty \omega(u) \frac{du}{(1+au)^{s+1}},
$$
where the integral converges for $\re(s)>0$. Writing
\begin{equation}\label{eqgaseval}
g_a(s)=s + \frac{e^{-\gamma}}{a (1+a)^s} 
+s \int_1^\infty \bigl( \omega(u) - e^{-\gamma}\bigr)   \frac{du}{(1+au)^{s+1}},
\end{equation}
shows that $g_a(s)$ is entire and allows us to evaluate $g_a(s)$ numerically.

Theorem \ref{rhoathm} is implied by the following result, which makes the error term more precise when $a=1/i$, $i=1,\ldots,10$.

\begin{theorem}\label{thmFara}
(i) For the values of $a$, $\lambda_a$, $\mu_a$ shown in Table \ref{table1}, $\widehat{F}_a(s)$ has a simple pole at $s=-\lambda_a$ and
no other poles with $\re(s)\ge -\mu_a$, and
\begin{equation}\label{rhoasy}
F_a(z)=C_a e^{-\lambda_a z} + O_a\left(e^{-\mu_a z}\right) \qquad (z\ge 0).
\end{equation}

(ii)
When $a>1$,  $\widehat{F}_a(s)$ has a simple pole at $s=-\lambda_a$ where $\lambda_a\in (0,1)$
and no other poles with $\re(s)\ge -\lambda_a -1- \varepsilon_a$, for some $\varepsilon_a>0$, and
$$
F_a(z)=C_a e^{-\lambda_a z} (1+O_a(e^{-z(1+\varepsilon_a)})) \qquad (z\ge 0).
$$

(iii)
 When $a>0$ and  $-\lambda_a$ is the maximal real part of the set of poles of  $\widehat{F}_a(s)$, then
$$
F_a(z) \asymp_a e^{-\lambda_a z}\qquad (z\ge 0).
$$
\end{theorem}

\begin{proof}
Part (i) is a consequence of Lemma \ref{FaInv} and the calculations shown in Table \ref{table1}.
Part (ii) is Lemma \ref{lemag1}.
Part (iii) is Lemma \ref{lemag0}.
\end{proof}

\begin{table}[h]
  \begin{tabular}{ | l | l | l | l | l |}
   \hline
    $a$ & $\lambda_a$ & $C_a$ & $\mu_a$ & complex poles \\ \hline
    1 & 1  & $\frac{1}{1-e^{-\gamma}}$ & 3.03 & $-3.03... \pm 11.36... i$ \\ \hline
1/2 & 2.46206... & 3.7815...  & 4.65 & $-4.65... \pm 18.71... i$ \\ \hline
1/3 & 4.20605... & 5.7645... & 6.50  & $-6.50... \pm 25.73... i$ \\ \hline
1/4 & 6.15900... & 8.3827... & 8.52   & $-8.53\ \ \pm 32.59... i$\\ \hline
1/5 & 8.27925... & 11.812... & 10.7  & $-10.7... \pm 39.3... i$ \\ \hline
1/6 & 10.5395... & 16.265... & 13.0 & $\approx -13 \pm 46 i$ \\ \hline
1/7 & 12.9203... &  22.000... & 15.3  & $\approx -15 \pm 52 i$ \\ \hline
1/8 & 15.4074... &  29.333... & 17.8  & $\approx -18 \pm 59 i$ \\ \hline
1/9 & 17.9892... &  38.648... &  20.2  & $\approx -20 \pm 66 i$ \\ \hline
1/10 & 20.6568... & 50.410... &  22.7  & $\approx -23\pm 72i$ \\ \hline
    \end{tabular}
  \caption{Values of $\lambda_a$,  $C_a$, $\mu_a$, and the approximate location of the pair of complex poles responsible for the oscillating secondary term.}
  \label{table1}
  \end{table}

The values of $\mu_a$ shown in Table \ref{table1}  are obtained with equation 
\eqref{gasomegaprime}, with the help of Mathematica and exact formulas for $\omega'(u)$ for $u\le 5$.
For $5<u\le 10.5533$ we replace $\omega'(u)$ by $(\omega(u-1)-\omega(u))/u$, 
and use a table of zeros and local extrema of $\omega(u)-e^{-\gamma}$ in \cite{CherGold}.
For $u>10.5533$ we use $|\omega'(u)|\le 1/\Gamma(u+1)$ from Lemma \ref{omegaest}.
It seems likely that $\mu_a = \lambda_a+2.03$ is admissible for all $0<a\le 1$. 

We estimate the residues $C_{a}$ by numerical evaluation of $\frac{1}{2\pi i}\int_{\Gamma_a} \widehat{F}_a(s)  ds$,
where $\Gamma_a$ is a closed contour around $-\lambda_{a}$.

\begin{lemma}\label{ga}
Let $a>0$.
We write $s=\sigma + i \tau$ and
$$ H_a(\sigma):=  \frac{1}{a (1+a)^\sigma} 
+\frac{1}{a} \int_1^\infty  |\omega'(t)| \frac{dt}{(1+at)^{\sigma}} .$$
If $g_a(s)=0$ then $|\tau| \le |s| \le H_a(\sigma)$. 
For $|s| \ge 2 H_a(\sigma)$, we have
$$ \frac{1}{g_a(s)}=\frac{1}{s} + O\left(\frac{H_a(\sigma)}{\sigma^2+\tau^2 } \right) .$$
\end{lemma}

\begin{proof}
Integration by parts shows that
\begin{equation}\label{gasomegaprime}
g_a(s) = s + \frac{1}{a (1+a)^s} 
+\frac{1}{a} \int_1^\infty  \omega'(t) \frac{dt}{(1+at)^{s}} 
 = s + H_a(\sigma) \xi_a(s),
\end{equation}
where $\xi_a(s) \in \mathbb{C}$ with  $|\xi_a(s)|\le 1$.
Thus any zero of $g_a(s)$ must satisfy $|\tau|\le |s| \le H_a(\sigma)$. 
We have
$$ g_a(s) = s \left(1+\frac{H_a(\sigma) \xi_a(s)}{s}\right),$$
from which the estimate for $1/g_a(s)$ follows.
\end{proof}

\begin{lemma}\label{FaInv}
Let $a>0$ and $\mu \ge 0$ be fixed. Let $F_a(z)$ be given by \eqref{FaDef} and 
let $P_\mu$ denote the finite set of poles of $\widehat{F}_a(s)$ with $\re(s)\ge -\mu$. 
Then 
$$
F_a(z) = \sum_{s_k \in P_\mu} \res(\widehat{F}_a(s) e^{zs}; s_k)+ O_{\mu,a}(e^{-\mu z}) \qquad (z\ge 0).
$$
\end{lemma}

In the following proofs, the constants implied by the symbols ``$\ll$" and ``$O$" may depend on $a$ and $\mu$.  

\begin{proof}
Since $0\le F_a(z) \le 1$,  the set of poles $P_\mu$ is contained in the half-plane $\re(s)\le 0$. 
Lemma \ref{ga} implies that $P_\mu$ is contained in a bounded region and is therefore finite. 
We evaluate the inverse Laplace integral 
$$ F_a(z) = \frac{1}{2\pi i} \int_{1-i\infty}^{1+i\infty} \widehat{F}_a(s) e^{zs}\, d s.$$
Let $T=\exp(z(\mu+1))$.
Since the result is trivial for bounded $z$, we may assume that $z$ is sufficiently large such that $T>2H_a(-\mu)$. We have
$$ \int_{1+iT}^{1+i\infty}  \widehat{F}_a(s) e^{zs}\, d s 
= \int_{1+iT}^{1+i\infty} \frac{1}{s} e^{zs}\, d s + O\left(e^{ z}/T \right)=O\left(e^{-\mu z}\right),$$
by Lemma \ref{ga} and integration by parts applied to the last integral.
We first  assume that $\widehat{F}_a(s)$ has no poles with $\re(s) =-\mu$. 
We apply the residue theorem to the rectangle with vertices $-\mu \pm iT$, $1\pm iT$. 
The contribution from the horizontal segments can be estimated by Lemma \ref{ga} as 
$$ \int_{-\mu+iT}^{1+iT} \widehat{F}_a(s) e^{zs}\,d s \ll \frac{e^{ z}}{T}  =O\left(e^{-\mu z}\right).$$
For the vertical segment with $\re s = -\mu$ we have
$$ \left| \int_{-\mu-i2H_a(-\mu)}^{-\mu+i2H_a(-\mu)} \widehat{F}_a(s) e^{zs}\,d s \right| \le 4H_a(-\mu) \max_{|\tau|\le 2H_a(-\mu)} \left|\widehat{F}_a(-\mu+i\tau)\right| e^{-\mu z} = O\left(e^{-\mu z}\right)$$
and
$$ \int_{-\mu+i2H_a(-\mu)}^{-\mu+iT} \widehat{F}_a(s) e^{zs}\, d s = \int_{-\mu+i2H_a(-\mu)}^{-\mu+iT} \left(\frac{1}{s}+O(\tau^{-2})\right) e^{zs}\,d s =O\left(e^{-\mu z}\right),$$
where the contribution from $\frac{1}{s}$ to the last integral can be estimated with integration by parts. The residue theorem now yields the result.
If $\widehat{F}_a(s)$ has a pole with $\re(s) =-\mu$, we use the modified rectangle with vertices  $-\mu -\varepsilon \pm iT$, $1 \pm iT$, 
for a suitable $\varepsilon >0$. 
\end{proof}

\begin{lemma}\label{lemag1}
Let $a>1$ and let $-\lambda_a$ denote the maximal real part of the set of poles of  $\widehat{F}_a(s)$.
Then $\widehat{F}_a(s)$ has a simple pole at the real number $s=-\lambda_a$, $\lambda_a \in (0,1)$, and 
 $\widehat{F}_a(s)$  has no other poles in the half-plane $\re(s) \ge -1 -\lambda_a -\varepsilon_a$, for some $\varepsilon_a>0$. 
Thus 
$$F_a(z) = C_a e^{-\lambda_a z}(1+O_a(e^{-z(1+\varepsilon_a)})) \qquad (z\ge 0).$$ 
\end{lemma}

\begin{proof}
The statement about the location and nature of the poles of $\widehat{F}_a(s)$ follows from \cite[Lemma 9]{SPA}.
The asymptotic estimate for $F_a(z)$ follows from Lemma \ref{FaInv}.
\end{proof}

\begin{lemma}\label{lemag0}
Let $a>0$ and let $-\lambda_a$ denote the maximal real part of the set of poles of  $\widehat{F}_a(s)$.
Then $\widehat{F}_a(s)$ has a simple pole at the real number $s=-\lambda_a$ and $$F_a(z) \asymp_a e^{-\lambda_a z} \qquad (z\ge 0).$$
\end{lemma}

\begin{proof}
Let $a>0$ be fixed and assume that $\widehat{F}_a(s)$ has poles with maximal real part $-\lambda$ at $s=-\lambda \pm i\tau_j $, $0\le j \le J$,
where $J\ge 0$ and $\tau_j \ge 0$. 
By Corollary \ref{corsimplepoles}, these poles are all simple. 
From Lemma \ref{FaInv} it follows that 
$$
F_a(z) = e^{-\lambda z} \left( o(1) + \alpha_0 + \sum_{j=1}^J(\alpha_j \cos \tau_j z + \beta_j \sin \tau_j z) \right) 
$$ 
for some real constants $\alpha_j$, $\beta_j$, $1\le j \le J$. We want to show that $F_a(z) \asymp_a e^{-\lambda z} $.
Since $F_a(z)=\rho_a(u)\ge 0$, where $e^z=1+au$, we must have $\alpha_0>0$. 
If there are no complex poles with real part $-\lambda$ (as is the case when $a\ge 1$ and for all $a$ shown in Table \ref{table1}, 
and probably also for all $a$ with $0<a<1$),
 then $F_a(z) \sim \alpha_0 e^{-\lambda z}$ as $z\to \infty$ and we are done. 
For the remainder we make the (unlikely) assumption that $J\ge 1$. 
Define $h:= \liminf_{z\to \infty} L(z)$, where
$$
L(z):=\alpha_0 + \sum_{j=1}^J(\alpha_j \cos \tau_j z + \beta_j \sin \tau_j z) .
$$
We want to show that $h>0$. 
If $h<0$ then clearly $F_a(z) <0$ for some $z$, which is impossible.
Assume that $h=0$, i.e.
there is a sequence $z_k$ with $\lim z_k = \infty$, such that $\lim L(z_k) = 0$.
Note that
$$
\int_{z_1}^{z_2} L(z) dz \ge \alpha_0 (z_2-z_1) -2 \sum_{j=1}^J \frac{|\alpha_j|+|\beta_j|}{\tau_j} \ge \frac{\alpha_0}{2}(z_2- z_1),
$$
provided 
$z_2-z_1\ge w$, where
$$w:= \frac{4}{\alpha_0} \sum_{j=1}^J \frac{|\alpha_j|+|\beta_j|}{\tau_j}.$$
It follows that, within any interval of length $w$,
$L(z)$ takes on a value greater than $\alpha_0/2$.
Since $\lim L(z_k) = 0$, we have $L(z_k) \le \alpha_0 e^{-\lambda w}/5$ for all large $k$. 
Also, $L(z_k+r_k) \ge  \alpha_0/2$ for suitable $r_k$ with $0\le r_k \le w$. Thus, for all large $k$,
$$
F_a(z_k+r_k)> e^{-\lambda (z_k+r_k)} (o(1) + \alpha_0/2) > e^{-\lambda (z_k+w)}\frac{\alpha_0}{3}
=  e^{-\lambda z_k} \frac{\alpha_0e^{-\lambda w} }{3}
$$
and
$$
F_a(z_k) =  e^{-\lambda z_k}(o(1)+L(z_k))<e^{-\lambda z_k}\frac{\alpha_0 e^{-\lambda w}}{4}< F_a(z_k+r_k).
$$
This contradicts the fact that $\rho_a(u)$ is decreasing in $u$, i.e. $F_a(z)$ is decreasing in $z$.
We have shown that $ \liminf_{z\to \infty} L(z)>0$, which implies the desired result for all sufficiently large $z$.
For bounded $z$, i.e. bounded $u$, the result follows from $1\ge \rho_a(u) \ge \rho(u)$. 
\end{proof}

\section{Proof of Theorem \ref{thmBaxyIntro}}\label{secBaxy}
Recall that  $\mathcal{B}_{a,y}$ (resp. $\mathcal{B}^*_{a,y}$)  is the set corresponding to $\theta(n) = y n^{a}$ (resp. $\theta(n) = \max(y, (yn)^{a})$) and $B_{a}(x,y)$ (resp. $B^*_{a}(x,y)$) is the corresponding counting function.
A tilde symbol on top of any of these sets or counting functions indicates that the set is restricted to squarefree numbers. 
We write
$$
u=\frac{\log x}{\log y}.
$$

\begin{theorem}\label{thmBaxy}
Let $a>0$ be fixed. Assume \eqref{rhoasy} holds for some constants $\mu_a \ge \lambda_a \ge 0$ and $\mu_a\neq \lambda_a+1$. 
Uniformly for $x\ge y\ge 2$, we have
$$
B_a(x,y) = x \, \eta_{a,y} \, \rho_a(u) +O_a\left(\frac{x u^{-\lambda_a}}{\log x} +\frac{x u^{-\mu_a}}{\log y}\right),
$$
$$
\tilde{B}_a(x,y) = x \, \tilde{\eta}_{a,y} \, \rho_a(u) +O_a\left(\frac{x u^{-\lambda_a}}{\log x} +\frac{x u^{-\mu_a}}{\log y}\right),
$$
where $\eta_{a,y} =1+O_a(1/\log y)$ and $\tilde{\eta}_{a,y} =1/\zeta(2)+O_a(1/\log y)$  for $y\ge 2$.

If  $\mu_a = \lambda_a+1$, both results hold with an error term of $O_a(x u^{-\lambda_a} \log(u)/\log x)$. 
\end{theorem}

Before proving Theorem \ref{thmBaxy} in the remainder of this section, we state and prove the following two Corollaries. 

\begin{corollary}\label{cor1Baxy}
Let $a>0$ be fixed. Uniformly for $x\ge y\ge 2$, 
$$
B_a(x,y) = x \, \rho_a(u) \left(1+O_a\left(\frac{1}{\log y} \right)\right) ,
$$
$$
\tilde{B}_a(x,y) = \frac{x \, \rho_a(u)}{\zeta(2)} \left(1+O_a\left(\frac{1}{\log y} \right)\right).
$$
We have
$$
B_a(x,y) \asymp x \, \rho_a(u)  \qquad (x\ge y\ge 2),
$$
$$
\tilde{B}_a(x,y) \asymp x \, \rho_a(u)  \qquad (x\ge y\ge y_0(a)).
$$
If  \eqref{rhoasy} holds for $\mu_a-\lambda_a >1$, we have
$$
B_a(x,y) = x \, \eta_{a,y} \, \rho_a(u) \left(1+O_a\left(\frac{1}{\log x} \right)\right)  \qquad (x\ge y\ge 2),
$$
$$
\tilde{B}_a(x,y) = x \, \tilde{\eta}_{a,y} \, \rho_a(u) \left(1+O_a\left(\frac{1}{\log x} \right)\right) \qquad (x\ge y\ge y_0(a)),
$$
where $\eta_{a,y} = 1+O_a(1/\log y)$ and $\eta_{a,y} \asymp 1$ for $y\ge 2$, while  $\tilde{\eta}_{a,y} =1/\zeta(2)+O_a(1/\log y)$;
by Table \ref{table1}, this holds for $a=1/i$, where $i=1,2,\ldots, 10$, and by Lemma \ref{lemag1} it holds for all $a>1$.
\end{corollary}

We conjecture that the last two estimates in Corollary \ref{cor1Baxy} hold for all $a>0$.
Since $D_1(x,y)=B_1(x,y)$ by Lemma \ref{lemTen}, the penultimate estimate of Corollary \ref{cor1Baxy} with $a=1$ coincides with \cite[Thm. 1.3]{PDD}.

\begin{proof}
The first two estimates follow from Theorem \ref{thmBaxy}, since $\rho_a(u)\asymp u^{-\lambda_a}$ and $\mu_a=\lambda_a$ is admissible 
by Lemma \ref{lemag0}. 
The third estimate follows from the first provided $y\ge y_0$, say.
If $2 \le y \le y_0$, we need to show $B_a(x,y) \gg_a x \rho_a(u)$.
We choose $m$ such that $2^{am}\ge y_0$. If $n$ is counted in $B_a(x/2^m, y_0)$, then $n'=2^m n$ is counted in $B_a(x,2)$.
Thus,
\begin{equation*}
B_a(x,y) \ge B_a(x,2) \ge B_a(x/2^m, y_0) \asymp x u^{-\lambda_a},
\end{equation*}
if $x/2^m \ge y_0$.
If $x/2^m < y_0$, both $B_a(x,y) \asymp 1$ and $x u^{-\lambda_a} \asymp 1$. 

The fourth estimate follows from the second.

The last two estimates follow from Theorem \ref{thmBaxy}. The claim that $\eta_{a,y} \asymp 1$ for $y\ge 2$ follows from $B_a(x,y) \asymp x \, \rho_a(u)$
when $y\ll 1$ .
\end{proof}

\begin{corollary}\label{corBaxy}
Let $a>0$ be fixed. We have
$$
B^*_{a} (x,y)  \asymp_a x  \rho_a(u)  \qquad (x\ge y\ge 2),
$$
$$
\tilde{B}^*_{a} (x,y)\asymp_a x  \rho_a(u)  \qquad (x\ge y\ge y_1(a)).
$$
\end{corollary}
\begin{proof}
When $a\ge 1$, then $\max(y, (yn)^{a}) = (yn)^a$ and the estimates follow from the order-of-magnitude results
in Corollary \ref{cor1Baxy}  with $y$ replaced by $y^a$.  

Hence we may assume $0<a<1$.

We clearly have $B^*_{a}(x,y) \le B_{a}(x,y)$.
To show  $B^*_{a}(x,y)\gg_a B_{a}(x,y)$, we consider two cases. First, assume that $y^a \ge 2$.
If $n=p_1\cdots p_r \in \mathcal{B}_{a,y^{a}}$, then 
$$
p_{j+1} \le y^{a} (p_1\cdots p_j)^{a} \le \max(y,(y p_1 \cdots p_j)^{a}) \qquad (0\le j < r).
$$
Thus, $\mathcal{B}^*_{a,y} \supseteq \mathcal{B}_{a,y^{a}}$ and 
$$
B^*_{a}(x,y) \ge B_{a}(x,y^{a}) \asymp_a x \left(\frac{\log y^a }{\log x}\right)^{\lambda_a}
\asymp_a x \left(\frac{\log y }{\log x}\right)^{\lambda_a}.
$$
Next, assume that $y^a<2$, and that $n=p_1\cdots p_r \in \mathcal{B}_{a,y}$. Let $n'=2^m n$, where $m=\lceil (1/a-1)/a \rceil$. 
Then $2 \le \max(y, y^a)=y$ and, for $0\le j<r$,
$$
p_{j+1} \le y (p_1 \cdots p_j)^a < (y 2^m p_1 \cdots p_j)^a \le \max(y,  (y 2^m p_1 \cdots p_j)^a).
$$
Thus, $n'=2^m n \in \mathcal{B}^*_{a,y}$ and 
$$
B^*_a(x,y) \ge B_a(x/2^m,y) \gg_a x \left(\frac{\log y }{\log x}\right)^{\lambda_a}.
$$

We clearly have $\tilde{B}^*_{a}(x,y) \le \tilde{B}_{a}(x,y)$. Since  $\max(y, (yn)^{a}) \ge y^a n^a$, 
$\tilde{B}^*_{a}(x,y) \ge \tilde{B}_a(x,y^a) \gg x \rho_a(u)$, for $y\ge y_1(a):=y_0(a)^{1/a}$, by Corollary \ref{cor1Baxy}.

\end{proof}

\begin{lemma}\label{lemhypeq}
Let $a>0$ be fixed.
Let $F_a(z)$ be given by \eqref{FaDef} and assume that 
\begin{equation}\label{hypeq1}
 F_a(z) =C_a e^{-\lambda_a z} + O_a(e^{-\mu_a z}) \qquad (z\ge 0),
\end{equation}
where $g_a(-\lambda_a) =0$ and $\mu_a \ge \lambda_a$. Then
\begin{equation}\label{hypeq2}
F'_a(z) = -\lambda_a C_a e^{-\lambda_a z} + O_a(e^{-\mu_a z}) \qquad (z\ge 0).
\end{equation}
\end{lemma}

\begin{proof}
If $0\le z < \log(a+1)$, we have $F_a(z)=1$ and hence $F'_a(z)=0$, while both the main term and error term of \eqref{hypeq2} are $O(1)$.
If $z\ge \log(a+1)$, \eqref{FaDef} implies
\begin{equation}\label{FapEq}
F'_a(z)= -\frac{1}{a} \int_0^{z-\log(a+1)} F_a(u) \, \omega'\left(\frac{e^{z-u}-1}{a}\right)  \frac{e^{z-u}}{a} \, d u 
 -\frac{1}{a} F_a(z-\log(a+1)).
\end{equation}
We apply \eqref{hypeq1} to each occurrence of $F_a$ on the right-hand side of \eqref{FapEq}.
Since the factor $\omega'(...)$ is $ \ll e^{-(z-u)(\mu+2)}$, the contribution from the error term in \eqref{hypeq1} to the integral is $O_a(e^{-\mu z})$.
For the contribution from the main term to the integral, apply integration by parts, using $\omega(v)-e^{-\gamma}$ as an antiderivative of $\omega'(v)$. 
The resulting integral, after the change of variables $\tilde{u}=z-u$, is 
$$
\int_{\log(1+a)}^z e^{-\lambda_a (z-u)} \left(\omega\left(\frac{e^u -1}{a}\right) - e^{-\gamma} \right) du,
$$
which we replace by 
$$\int_0^\infty - \int_z^\infty - \int_0^{\log(1+a)}.$$
The first integral equals
$$
e^{-\lambda_a z} \left(\widehat{\Omega}_a(-\lambda_a) + \frac{e^{-\gamma}}{\lambda_a}\right) =e^{-\lambda_a z} \left(-a + \frac{e^{-\gamma}}{\lambda_a}\right),
$$
since $g_a(-\lambda_a)=0$. The second integral is $O_a(e^{-\mu z})$. The third integral can be evaluated easily, since the Buchstab function vanishes here.
Combining everything yields \eqref{hypeq2}.
\end{proof}

Let $\mathcal{B}_\theta$ be as in Section \ref{subsecBtheta} and let $B_\theta(x)$ be its counting function. Let
\begin{equation}\label{thetadef}
\theta: \mathbb{N}\to \mathbb{R}\cup \{\infty\}, \quad 2\le \theta(n)\le \theta(nm) ,\quad (n,m \in \mathbb{N}).
\end{equation} 
Define
$$
\Phi(x,y) = |\{1\le n\le x: P^-(n)>y\}|.
$$
The following two lemmas are from \cite{SPA}, where it is assumed that $\theta(n) \ge P^+(n)$ for all $n\ge 1$. 
If  $n\in \mathcal{B}_\theta$, \eqref{thetadef} implies $\theta(n)\ge P^+(n)$, since 
$\theta(n) \ge \theta(n/P^+(n)) \ge P^+(n). $ If $n \notin \mathcal{B}_\theta$, the value of $\theta(n)$ does not affect $\mathcal{B}_\theta$. 
Thus the following two lemmas also hold assuming \eqref{thetadef}.

\begin{lemma}\label{Lem1}(see \cite[Lemma 3]{SPA})
Assume $\theta$ satisfies \eqref{thetadef}. For $x\ge 0$, 
$$
[x] = \sum_{n \in \mathcal{B}_\theta} \Phi(x/n, \theta(n)).
$$
\end{lemma}

\begin{lemma}\label{Sum1}(see \cite[Theorem 1]{SPA})
Assume $\theta$ satisfies \eqref{thetadef}.
If $B_\theta(x)=o(x)$, then 
$$ 1=\sum_{n \in \mathcal{B}_\theta} \frac{1}{n} \prod_{p\le \theta(n)} \left(1-\frac{1}{p}\right).$$
\end{lemma}

\begin{lemma}\label{PhiLemma}(see \cite[Lemma 6]{OMG})
Uniformly for $x>0$, $y\ge 2$, $u=\frac{\log x}{\log y}$,
\begin{equation*}
\begin{split}
\Phi(x,y) &= 1_{x\ge 1}+ x \prod_{p\le y}\left(1-\frac{1}{p}\right)
+\frac{x}{\log y}\! \left\{\omega(u)-e^{-\gamma}\! -\left.\frac{y}{x}\right|_{x\ge y} 
+ O\left(\frac{e^{-u/3}}{\log y}\right)\! \right\}.\\
\end{split}
\end{equation*}
\end{lemma}

\begin{lemma}\label{LemFinalSum}
Assume $\theta(n) = y n^a$ where  $a>0$ is constant and write $u=\log x / \log y$. 
Let $\mathcal{B}=\mathcal{B}_{a,y}$ and $B(x)=B_a(x,y)$.
Assume $B(x) \ll x/ (u+1)^r$ for some constant $r\ge 0$, uniformly for $x\ge1, y\ge 2$. Then 
$$
B(x) = x\sum_{n\in \mathcal{B}}\frac{1}{n \log y n^a} \left( e^{-\gamma} - \omega\left(\frac{\log x/n}{\log y n^a}\right)\right)  
+O_a\left(1+ \frac{x (\log y)^r}{(\log xy)^{r+1}}\right),
$$
uniformly for $x\ge 1$, $y\ge 2$. 
\end{lemma}

\begin{proof}
We substitute the estimate in Lemma \ref{PhiLemma} into Lemma \ref{Lem1}. With Lemma \ref{Sum1}, this yields
\begin{multline*}
B(x) = x\sum_{n\in \mathcal{B}}\frac{1}{n \log \theta(n)} \left( e^{-\gamma} - \omega\left(\frac{\log x/n}{\log \theta(n)}\right)\right)  \\
 [x]-x +\sum_{n\in \mathcal{B} \atop n\theta(n) \le x} \frac{\theta(n)}{\log \theta(n)}
+O\left(\sum_{n\in \mathcal{B}} \frac{x}{n \log^2 \theta(n)} \exp\left( -\frac{\log x/n}{3\log \theta(n)}\right)\right).
\end{multline*}

We have $[x]-x=O(1)$ and
\begin{equation*}
\sum_{n\in \mathcal{B} \atop n^{1+a}\le x/y} \frac{ y n^a}{\log yn^a} 
 \ll \frac{y (x/y)^\frac{a}{1+a}}{\log \left( y (x/y)^\frac{a}{1+a}\right) }B((x/y)^{\frac{1}{1+a}})
\ll_a \frac{x (\log y)^r}{(\log xy)^{r+1}}.
\end{equation*}
Let $A:=3\max(1,a)$. Partial summation yields
\begin{equation*}
\begin{split}
\sum_{n\in \mathcal{B} }\frac{x \exp\left(-\frac{\log x/n}{3\log yn^a}\right) }{n (\log yn^a)^2} 
 & \ll_a \int_{ 1}^\infty \frac{x (\log y)^r \exp\left(-\frac{\log xy}{A\log yt}\right)}{t (\log yt)^{2+r}}dt  
 \ll \frac{x(\log y)^r}{(\log xy)^{1+r}},
\end{split}
\end{equation*}
where the last estimate follows from the change of variables $v=(\log xy)/\log yt$. 
\end{proof}

\begin{lemma}\label{lem4}
With the notation and assumptions of Lemma \ref{LemFinalSum}, we have
\begin{equation*}
B(x)= x\int_{1}^{\infty} \frac{B(t)}{ t^2 \log y t^a} 
\left( e^{-\gamma} - \omega\left(\frac{\log x/t}{\log y t^a}\right)\right) \, dt
 + O_a\left(1+\frac{x (\log y)^r}{(\log x y)^{1+r}}\right),
\end{equation*}
uniformly for $x\ge 1$, $y\ge 2$. 
\end{lemma}

\begin{proof}
This result follows from applying partial summation to the sum in Lemma \ref{LemFinalSum}. 
\end{proof}

\begin{proof}[Proof of the first estimate in Theorem \ref{thmBaxy}]
From Lemma \ref{lem4} we have, for $x\ge 1$, $y\ge 2$, 
\begin{equation}\label{inteq}
B(x)= x \, \alpha_y - x \int_{1}^{\infty} \frac{B(t)}{t^2 \log y t^a} 
\, \omega\left(\frac{\log x/t}{\log y t^a}\right) d t
 +  O\left(1+\frac{x (\log y)^r}{(\log xy)^{1+r}}\right),
\end{equation}
where
$$ \alpha_y:=e^{-\gamma} \int_{1}^{\infty} \frac{B(t)}{t^2 \log y t^a} \, d t . $$
Lemma \ref{Sum1} and partial summation shows that 
\begin{equation}\label{alphaestimate}
\alpha_y \ll 1.
\end{equation}
For $x\ge 1$, $t\ge 1$, define $z\ge 0, v\ge 0$ by  
$$e^z=\frac{\log (y x^a)}{\log y}, \quad e^v =\frac{\log(y t^a)}{\log y} $$
and let
$$ G_y(z) := \frac{B(x)}{x} .$$
Dividing \eqref{inteq} by $x$ and changing variables in the integral we get, for $z\ge 0$,
\begin{equation}\label{conv}
\begin{split}
G_y(z) & =  \alpha_y - \frac{1}{a} \int_{0}^{z} G_y(v) \, \omega\left((e^{z-v} -1)/a\right) \, d v +E_y(z) \\
       & = \alpha_y -\frac{1}{a}  \int_{0}^{z} G_y(v) \, \Omega_a(z-v)  \, d v +E_y(z), 
\end{split} 
\end{equation}
where 
\begin{equation}\label{Error}
E_y(z) \ll  y^{-(e^z-1)/a}+e^{-(1+r)z} (\log y)^{-1} \ll e^{-(1+r)z}
\end{equation}
and
$$ \Omega_a(v):= \omega\left((e^{v} -1)/a\right) . $$
Now multiply \eqref{conv} by $e^{-zs}$, where $s\in \mathbb{C}$, $\re s > 0$, and integrate over $z\ge 0$ to obtain the equation of Laplace transforms
$$
 \widehat{G}_y(s) = \frac{\alpha_y}{s} -\frac{1}{a} \widehat{G}_y(s) \, \widehat{\Omega}_a(s) + \widehat{E}_y(s) \qquad (\re s >0).
$$
Solving for $\widehat{G}_y(s)$, we get
\begin{equation*}\label{LaplaceEq}
\widehat{G}_y(s) = \frac{\alpha_y}{s (1+\frac{1}{a}\widehat{\Omega}_a(s))} + \frac{\widehat{E}_y(s)}{1+\frac{1}{a}\widehat{\Omega}_a(s)}
\qquad (\re s >0).
\end{equation*}
By equation \eqref{FaLap},
\begin{equation}\label{LaplaceEq2}
\widehat{G}_y(s) =  \alpha_y \widehat{F}_a(s) + s  \widehat{F}_a(s)\widehat{E}_y(s)
\qquad (\re s >0).
\end{equation}

Since $\widehat{F'_a}(s)=s\widehat{F}_a(s)-F_a(0)$ and $F_a(0)=1$, this yields
$$    \widehat{G}_y(s) = \alpha_y \widehat{F}_a(s) + \widehat{E}_y(s) +\widehat{F'_a}(s) \widehat{E}_y(s),$$
and thus
\begin{equation}\label{last}
G_y(z)= \alpha_y F_a(z) + E_y(z) + \int_0^z F'_a(z-v) E_y(v) dv.
\end{equation}
With Lemma \ref{lemag0} and estimates \eqref{hypeq2}, \eqref{alphaestimate} and \eqref{Error}, equation \eqref{last} yields
$$
G_y(z) \ll_a e^{-\lambda_a z} + e^{-(1+r)z} + \int_0^z e^{-\lambda_a(z-v)} e^{-(1+r)v} dv.
$$
Since $B(x) \le x$, $G_y(z) \le 1$ and $r=0$ is admissible, starting with Lemma \ref{LemFinalSum}. 
Applying the preceding argument repeatedly shows that 
$$
G_y(z) \ll_a e^{-\lambda_a z},
$$
so that $r=\lambda_a$ is admissible. 
For the remainder of the proof we choose $r=\lambda_a$ and assume that $x\ge y$, that is $e^z \ge 1+a$. 
(This assumption was 
not permitted earlier because we integrated over $z\ge 0$ when finding Laplace transforms.)

We estimate $F_a(z)$ and $F'_a(z-v)$ with \eqref{hypeq1} and \eqref{hypeq2} of Lemma \ref{lemhypeq}, and estimate $E_y(z)$ and $E_y(v)$ with
 the first upper bound of \eqref{Error}.
The contribution to the integral in \eqref{last} from the main term in \eqref{hypeq2} is
\begin{equation*}
\begin{split}
& \int_0^z -\lambda_a C_a  e^{-\lambda_a (z-v)} E_y(v) dv \\
&=  -\lambda_a C_a   e^{-\lambda_a z} \left(  \int_0^\infty e^{\lambda_a v} E_y(v) dv 
+ O\left( \int_z^\infty e^{\lambda_a v}E_y(v) dv \right) \right) \\
&= -\lambda_a C_a     e^{-\lambda_a z}\beta_y + O\left(e^{-z(\lambda_a+1)}(\log y)^{-1})\right) \\
& = -\lambda_a \beta_y F_a(z) + O(e^{-\mu_a z}\beta_y) + O\left(e^{-z(\lambda_a+1)}(\log y)^{-1}\right),
\end{split}
\end{equation*}
by \eqref{hypeq1}, where
 $$\beta_y :=  \int_0^\infty e^{\lambda_a v} E_y(v) dv  \ll 1/\log y . $$
The contribution from the error term in \eqref{hypeq2} to the integral in \eqref{last} is
$$
\ll \int_0^z e^{-\mu_a(z-v)} E_y(v) dv 
\ll e^{-\min(\mu_a,1+\lambda_a) z}  (\log y)^{-1} ,
$$
provided $\mu_a \neq 1+\lambda_a$. If  $\mu_a = 1+\lambda_a$, the error term needs an extra factor of $z \asymp_a \log u$.
Hence \eqref{last} implies 
$$ G_y(z)= (\alpha_y - \lambda_a \beta_y) F_a(z) + O(e^{-\min(\mu_a,1+\lambda_a) z}  (\log y)^{-1} ).
$$
Writing
$$
\eta_{a,y}:=\alpha_y - \lambda_a \beta_y,
$$
we obtain
$$
B(x) = B_a(x,y) = x\eta_{a,y} \rho_a(u) + O(x u^ {-\min(\mu_a,1+\lambda_a) }(\log y)^{-1} ).
$$
When $y=x$, we have $B_a(x,y)= \lfloor x \rfloor$ and $\rho_a(u)=\rho_a(1)=1$. Thus the last display implies that
$$
\eta_{a,y}=1+O(1/\log y).
$$
This completes the proof of the first estimate in Theorem \ref{thmBaxy}.
\end{proof}

\begin{proof}[Proof of the second estimate in Theorem \ref{thmBaxy} (the squarefree case)]
Define
$$
\Phi_0(x,y) =| \{1\le n \le x: P^-(n)>y \text{ and } \mu^2(n)=1\}|.
$$
The number of squarefree integers up to $x$ is (see \cite[Exercise 44]{Ten})
\begin{equation}\label{SFQ}
\Phi_0(x,1) = \frac{x}{\zeta(2)}+O(\sqrt{x}).
\end{equation}
Let $\tilde{\mathcal{B}}_\theta$ be the set of squarefree integers in $\mathcal{B}_\theta$ and let $\tilde{B}_\theta(x)$ be the corresponding counting function. 
The following three lemmas are the squarefree analogues of Lemmas \ref{Lem1}, \ref{Sum1} and \ref{PhiLemma}.
Different versions of these lemmas appear in \cite[Lemmas 2.2, 2.3, 3.3, 3.5]{PTW}.
\begin{lemma}\label{Lem1sf}
Assume $\theta$ satisfies \eqref{thetadef}. For $x\ge 0$, 
$$
\Phi_0(x,1) = \sum_{n \in \tilde{\mathcal{B}}_\theta} \Phi_0(x/n, \theta(n)).
$$
\end{lemma}
\begin{proof}
Each $m$ counted in $\Phi_0(x,1)$ factors uniquely as $m=nr$, where $n\in \tilde{\mathcal{B}}_\theta$, $P^-(r)>\theta(n)$, $\mu^2(r)=1$ and $r\le x/n$. 
\end{proof}

\begin{lemma}\label{Sum1sf}
Assume $\theta$ satisfies \eqref{thetadef}.
If $\tilde{B}_\theta(x)=o(x)$, then 
$$ 1=\sum_{n \in \tilde{\mathcal{B}}_\theta} \frac{1}{n} \prod_{p\le \theta(n)} \left(1+\frac{1}{p}\right)^{-1}.$$
\end{lemma}
\begin{proof}
This follows from Lemma \ref{Lem1sf}, the estimate \eqref{SFQ}, the upper bound (see Lemma \ref{Phi0Lemma})
$\Phi_0(x,y) \le 1+x/\log y$ for $x\ge 1$, $y\ge 2$, and (see Lemma \ref{Phi0Lemma}) for fixed $y$,
$$
\lim_{x\to \infty}\Phi_0(x,y)/x = \frac{1}{\zeta(2)}\prod_{p\le y} \left(1+\frac{1}{p}\right)^{-1}.
$$
\end{proof}

\begin{lemma}\label{Phi0Lemma}
Uniformly for $x>0$, $y\ge 2$, $u=\frac{\log x}{\log y}$,
\begin{multline*}
\Phi_0(x,y) = \\ 1_{x\ge 1}+ \frac{x}{\zeta(2)} \prod_{p\le y}\left(1+\frac{1}{p}\right)^{-1} 
+\frac{x}{\log y} \left\{\omega(u)-e^{-\gamma} -\left.\frac{y}{x}\right|_{x\ge y} 
+ O\left(\frac{e^{-u/3}}{\log y}\right)\right\}.
\end{multline*}
\end{lemma}
\begin{proof}
We combine the ideas of the proofs of \cite[Lemmas 2.2 and 2.3]{PTW} with the estimate for $\Phi(x,y)$ in Lemma \ref{PhiLemma}:
When $\log^2 y > \log x$, we have
$$
0\le \Phi(x,y)-\Phi_0(x,y) \le \sum_{p>y}\sum_{n\le x \atop p^2|n}1 \le \sum_{p>y}\left\lfloor \frac{x}{p^2} \right\rfloor 
\ll \frac{x}{y \log y} \ll \frac{x e^{-u/3}}{\log^2 y}
$$
and
$$
0\le x \prod_{p\le y}\left(1-\frac{1}{p}\right) - \frac{x}{\zeta(2)} \prod_{p\le y}\left(1+\frac{1}{p}\right)^{-1} 
\ll \frac{x}{y \log^2 y} \ll \frac{x e^{-u/3}}{\log^2 y},
$$
so that the result follows from Lemma \ref{PhiLemma}.

When $\log^2 y \le \log x$, we write
$$
\Phi_0(x,y)=\sum_{n\le x \atop P^{-}(n) >y} \mu^2(n) = \sum_{n\le x \atop P^{-}(n)>y}\sum_{d^2|n}\mu(d) 
=\sum_{d\ge 1 \atop  P^{-}(d)>y}\mu(d) \Phi(x/d^2,y),
$$
and the result follows from inserting the estimate in Lemma \ref{PhiLemma} into the last sum. Indeed, the contribution from the term 
$1_{x\ge 1}$ in Lemma \ref{PhiLemma} is $\ll \sqrt{x} \ll x e^{-u/3}/\log^2 y$. 
The contribution from the term with the Euler product in  Lemma \ref{PhiLemma} is exactly the term with the Euler product in Lemma \ref{Phi0Lemma}.
The contribution from the remaining term in Lemma \ref{PhiLemma} and $d=1$ matches the remaining term in Lemma \ref{Phi0Lemma},
while $d>1$ (i.e. $d>y$) contributes $\ll  x e^{-u/3}/\log^2 y$, since $\omega(u)-e^{-\gamma} \ll e^{-u}$. 
\end{proof}

Inserting the estimate in Lemma \ref{Phi0Lemma} for $\Phi_0(x/n, y n^a)$ into Lemma \ref{Lem1sf}, and the estimate \eqref{SFQ} for
$\Phi_0(x,1)$, 
we find that Lemma \ref{LemFinalSum} holds with $B(x)$ replaced by $\tilde{B}(x)$, $\mathcal{B}$ replaced by $\tilde{\mathcal{B}}$, 
$r$ chosen as $\lambda_a$ (since $\tilde{B}(x) \le B(x) \ll x/u^{\lambda_a}$), and the error term $O(1)$ replaced by $O(\sqrt{x})$ (from \eqref{SFQ}). 

The remainder of the proof is the same as with $B(x)$, but with the new error term $O(\sqrt{x})$ replacing $O(1)$, which has no effect on the
final outcome. As in the case of $B(x)$, we arrive at
$$
\tilde{B}(x) = \tilde{B}_a(x,y) = x\tilde{\eta}_{a,y} \rho_a(u) + O(x u^ {-\min(\mu_a,1+\lambda_a) }(\log y)^{-1} ).
$$
When $y=x$, we have $\tilde{B}_a(x,y)= \Phi_0(x,1) = x/\zeta(2) +O(\sqrt{x})$ and $\rho_a(u)=\rho_a(1)=1$. Thus the last display implies that
$$
\tilde{\eta}_{a,y}=1/\zeta(2)+O(1/\log y).
$$
This completes the proof of Theorem \ref{thmBaxy}.
\end{proof}

\section{Proof of Theorem \ref{thmAbeta}}\label{secAbeta}

If $1\le x \le y$ and $y\ge 2$, then $B_a(x,y)=\lfloor x \rfloor \asymp x$.
Theorems \ref{rhoathm} and \ref{thmBaxyIntro} imply
$$
B_a(x,y) \asymp_a x \left(\frac{\log y}{\log xy} \right)^{\lambda_a} \qquad (x\ge 1, \ y\ge 2),
$$
while $B_{a}(x,y) =1_{x\ge 1}$ if $y<2$. 
Similarly, in the squarefree case, we have
$$
\tilde{B}_a(x,y) \asymp_a x \left(\frac{\log y}{\log xy} \right)^{\lambda_a} \qquad (x\ge 1, \ y\ge y_0(a)).
$$
For $x\ge 1$, $y\ge 1$, $y_0=y_0(1/\beta)$, we have
\begin{multline*}
\tilde{A}_\beta(y_0^\beta x, y) \ge |\{n\le x: F_\beta(n)\le y_0^\beta y n, \mu^2(n)=1\}| \\
=\tilde{B}_{1/\beta}(x,y_0 y^{1/\beta}) \asymp_\beta x  \left(\frac{\log 2y}{\log 2xy} \right)^{\lambda_{1/\beta}} ,
\end{multline*}
where the equality is due to \eqref{eqSSFcount}.
This implies the desired lower bound.

For the upper bound, we write
\begin{equation*}
\begin{split}
A_\beta(x,y) &= \sum_{k\ge 0 } |\{ x 2^{-k-1} < n \le x 2^{-k}: F_\beta(n) \le x y\}| \\
 & \le \sum_{k\ge 0 } |\{ x 2^{-k-1} < n \le x 2^{-k}: F_\beta(n) \le n y 2^{k+1}\}| \\
 & \le \sum_{k\ge 0 } B_{1/\beta}(x/2^k, (y 2^{k+1})^{1/\beta} ) \\
 & \ll_\beta \sum_{k\ge 0} \frac{x}{2^k} \left(\frac{\log y 2^{k+1} }{\log 2 xy}\right)^{\lambda_{1/\beta}}\\
  & \ll_\beta  x  \left(\frac{\log 2y}{\log 2xy} \right)^{\lambda_{1/\beta}}.
\end{split}
\end{equation*}

\section{Proof of Theorem \ref{thmlambda}}\label{seclambda}

Recall that $-\lambda_a$ is the real zero of $g_a(s)$ with maximal real part. 
Theorem \ref{thmlambda} follows from the lower bound for $\lambda_a$ in Corollary \ref{corlambdalb} and the upper bound in Corollary \ref{corlambdaub}.

The following Lemma provides an alternative for calculating $g_a(s)$, defined in \eqref{eqgaseval}, which does not involve Buchstab's function $\omega(u)$. 
\begin{lemma}\label{lemgJ}
Let $a>0$. 
For $\re(s)>0$ we have 
$$
g_a(s) = \frac{1}{\Gamma(s)} \int_0^\infty u^s e^{-u+J(au)} du = \frac{1}{a^{s+1} \Gamma(s)} \int_0^\infty u^s e^{-u/a+J(u)} du  ,
$$
where, for $u>0$,
$$
J(u):=\int_u^\infty \frac{e^{-t}}{t} \, dt.
$$
\end{lemma}

\begin{proof}
Define $f(v)=\omega((v-1)/a)$.  For  $\re(s)>0$, we have
$$\hat{f}(s) = a e^{-s} \hat{\omega}(as) =a e^{-s}(e^{J(as)}-1),$$ 
by \cite[Eq. III.6.36]{Ten},
and 
\begin{equation*}
\begin{split}
\widehat{\Omega}_a(s) &=
\int_0^\infty \omega((e^u-1)/a) e^{-us} du  \\
&= \int_0^\infty \omega((v-1)/a) v^{-s-1} dv  \\
&=\int_0^\infty f(v) v^{-s-1} dv  \\
&=\frac{1}{ \Gamma(s+1)}\int_0^\infty f(v)  \int_0^\infty (w/v)^s e^{-w} v^{-1} dw dv  \\
&=\frac{1}{ \Gamma(s+1)}\int_0^\infty f(v)  \int_0^\infty u^s e^{-uv} du dv  \\
&= \frac{1}{ \Gamma(s+1)} \int_0^\infty u^s \int_0^\infty f(v)   e^{-uv} dv du  \\
&= \frac{1}{ \Gamma(s+1)} \int_0^\infty u^s  \hat{f}(u) du  \\
&= \frac{1}{ \Gamma(s+1)} \int_0^\infty u^s  a e^{-u}(e^{J(au)}-1) du  \\
&= \frac{a}{ \Gamma(s+1)} \int_0^\infty u^s   e^{-u + J(au)}du -a.  \\
\end{split}
\end{equation*}
The result now follows since $g_a(s):=1/\widehat{F}_a(s)=s(1+\frac{1}{a}\widehat{\Omega}_a(s) )$, by \eqref{FaLap}. 
\end{proof}

\begin{corollary}\label{cordgda}
For real $a>0$ and all $s\in \mathbb{C}$ we have
$$
\frac{\partial }{ \partial a} g_a(s) = \frac{s}{a} g_a(s+1)-\frac{s+1}{a} g_a(s).
$$
\end{corollary}

\begin{proof}
Note that \eqref{eqgaseval} implies that for every $a>0$, $g_a(s)$ and $\frac{\partial}{\partial a} g_a(s) $ are both entire functions of $s$. 
When $\re(s)>0$, the result follows from differentiating, with respect to $a$, the second expression for $g_a(s)$ in Lemma \ref{lemgJ}.
The result is valid for all $s\in \mathbb{C}$ by analytic continuation.
\end{proof}

\begin{corollary}\label{corsimplepoles}
Let $a>0$. Every zero of $g_a(s)$ with real part $>-\lambda_a - 1$ is simple. 
\end{corollary}
\begin{proof}
Let $s_0(a)$ denote a zero of $g_a(s)$ with $\re(s_0(a))>-\lambda_a - 1$. Then $g_a(s_0(a))=0$ and 
$$
0=\frac{d}{da} g_a(s_0(a)) = \left.\left( \frac{\partial }{\partial a} g_a(s) \right) \right|_{s=s_0(a)} + g_a'(s_0(a)) \frac{d s_0(a)}{da},
$$
where $g_a'(s):=\frac{\partial}{\partial s}g_a(s)$. With Corollary \ref{cordgda}, we obtain
$$
0= \frac{s_0(a)}{a} g_a(s_0(a)+1) + g_a'(s_0(a)) \frac{d s_0(a)}{da}.
$$
The first term on the right-hand side is $\neq 0$, since $\re(s_0(a)+1)>-\lambda_a$, and $s_0(a)\neq 0$ by \eqref{eqgaseval}.
Thus $ g_a'(s_0(a))\neq 0$, which is the desired conclusion.
\end{proof}

To evaluate $g_a(s)$ with Lemma \ref{lemgJ}, we write
$$
 h_a(s):=\int_0^\infty u^s e^{-u/a+J(u)} du =\int_0^1  +\int_1^\infty =: h_{1,a}(s)+h_{2,a}(s),
$$
say. Note that $h_{2,a}(s)$ is entire and can be evaluated with high precision using numerical integration.
Writing 
$$
I(s):= \int_0^s \frac{e^t-1}{t} dt,
$$
we have $J(u)=-\gamma -\log u - I(-u)$, by \cite[Lemma III.5.9]{Ten}, and
$$
e^\gamma h_{1,a}(s) =  \int_0^1 u^{s-1 }e^{-u/a} e^{-I(-u)} du =  \int_0^1 u^{s-1 }e^{-u/a} \sum_{k\ge 0} b_k u^k  du,
$$
where 
$$
e^{-I(-u)}= \sum_{k\ge 0} b_k u^k .
$$
By Cauchy's inequality,
$$
|b_k| \le  2^{-k}\max_{|s|=2} | e^{-I(-s)}| < \frac{4}{2^k}.
$$
Thus, for $\re(s)>1$, we have
$$
e^\gamma h_{1,a}(s) = \sum_{k\ge 0} b_k  \int_0^1 u^{k+s-1 }e^{-u/a}  du = \sum_{k\ge 0} a^{k+s} b_k  \int_0^{1/a} t^{k+s-1 }e^{-t}  dt.
$$
With the incomplete gamma function $\Gamma(s,z):=\int_z^\infty t^{s-1} e^{-t} dt$, this yields
$$
e^\gamma h_{1,a}(s) =\sum_{k\ge 0} a^{k+s} b_k (\Gamma(s+k)-\Gamma(s+k,1/a)).
$$
When $\re(k+s-1)\ge 0$, i.e. $k\ge 1-\re(s)$, then 
$$
|\Gamma(s+k)-\Gamma(s+k,1/a)| \le \int_0^{1/a} (1/a)^{k+\re{s}-1} e^{-t} dt <  (1/a)^{k+\re{s}-1}.
$$
Thus the last series converges absolutely and provides the meromorphic continuation of $h_{1,a}(s)$ to the entire complex plane. 
Moreover, for $K\ge 1-\re(s)$,
$$
\left| h_{1,a}(s) -e^{-\gamma}\sum_{k= 0}^K a^{k+s} b_k (\Gamma(s+k)-\Gamma(s+k,1/a))\right| 
<  a e^{-\gamma}\sum_{k>K} |b_k| < \frac{4a e^{-\gamma}}{2^K}.
$$
Writing 
$$
h_{1,a,K}(s):=e^{-\gamma}\sum_{k= 0}^K a^{k+s} b_k (\Gamma(s+k)-\Gamma(s+k,1/a)),
$$
it follows that 
$$
|h_a(s) - h_{1,a,K}(s) - h_{2,a}(s)| <  \frac{4a e^{-\gamma}}{2^K}.
$$
Since $g_a(s) = \frac{h_a(s)}{a^{s+1} \Gamma(s)}$, we obtain the explicit error estimate for approximating $g_a(s)$, valid for all complex $s$, 
real $a>0$ and for all $K\ge 1-\re(s)$:
\begin{equation}\label{gasapprox}
\left|g_a(s) - \frac{h_{1,a,K}(s) + h_{2,a}(s)}{a^{s+1} \Gamma(s)} \right| <   \frac{4 e^{-\gamma}}{2^K a^{\re(s)} |\Gamma(s)|}.
\end{equation}

Since $\Gamma(s)$ has poles at integers $\le 0$ and $\Gamma(s,1/a)$ is entire, the last display implies that for all non-negative integers $n$,
\begin{equation}\label{gapoly}
g_a(-n) a e^\gamma = \sum_{k=0}^n a^k b_k \prod_{0\le j \le k-1} (-n+j) =  \sum_{k=0}^n (-a)^k b_k \frac{n!}{(n-k)!} .
\end{equation}
Thus $g_a(-n) a e^\gamma$ is a polynomial in $a$ of degree $n$ with rational coefficients, since the $b_k$ are rational.
This formula gives a simple way of finding the exact values of $g_a(-n)$. We have $g_a(0) a e^\gamma =1 > 0$. The smallest integer $n\ge 0$ for which $g_a(-n-1)<0$
is the likely value of $\lfloor \lambda_{1/a} \rfloor$. 

Writing $f(s):=e^{-I(-s)}$, another way to write this formula is 
$$
g_a(-n) a e^\gamma = \sum_{k=0}^n (-a)^k b_k k! \binom{n}{k}=   \sum_{k=0}^n (-a)^k f^{(k)}(0)  \binom{n}{k} 
=\left. \frac{d^n}{ds^n} \left( e^s f(-as) \right) \right|_{s=0}.
$$
Define $L_a(s) = e^s f(-as) = e^{s-I(as)}$. Then, for $n\in \mathbb{N}\cup\{0\}$,
$$
g_a(-n) a e^\gamma =L_a^{(n)}(0) = \frac{n!}{2\pi i} \int_{\mathcal{C}} \frac{L_a(s)}{s^{n+1}} ds =  \frac{n!}{2\pi i} \int_{\mathcal{C}}  \frac{e^{s-I(as)}}{s^{n+1}} ds,
$$
where ${\mathcal{C}} $ is a positively oriented contour around the origin. 
We have shown the following:

\begin{lemma}
For real $a>0$ and integers $n\ge 0$ we have
$$
g_a(-n) a e^\gamma = \frac{n!}{2\pi i} \int_{\mathcal{C}}  \frac{e^{s-I(as)}}{s^{n+1}} ds,
$$
where ${\mathcal{C}} $ is a positively oriented contour around the origin. 
\end{lemma}

To learn about the zeros of $g_a(s)$, we will estimate the last integral with the saddle point method.
Writing $s=r e^{i t}$, we have
\begin{equation}\label{gacint}
 g_a(-n) a e^\gamma 
= \frac{n!}{2\pi } \frac{e^{r-I(ar)}}{r^n}  \int_{-\pi}^\pi \exp\left(M_{a,r,n}(t)\right)dt,
\end{equation}
for integers $n\ge 0$, where
$$
M_{a,r,n}(t):=r e^{it} - r -I(are^{it})+I(ar) -i t  n .
$$
The derivative of $M_{a,r,n}(t)$ with respect to $t$ vanishes when $t=0$ and 
\begin{equation}\label{nrdef}
n= r-e^{ar}+1 .
\end{equation}
In terms of the Lambert $W$ function, this equation has the solution
\begin{equation}\label{rdef}
r=r(a,n)=n-1-\frac{1}{a} W(-a e^{a(n-1)}).
\end{equation}
The value of $W(u)$ is real if and only if $u\ge -1/e$. Thus $r(a,n)$ is real when $a e^{a(n-1)} \le 1/e$, that is 
$$
n \le n_0:= 1 + \frac{1}{a} \log \frac{1}{ea}.
$$
When $1\le n \le n_0$ and $0<a<1/e$, we note that $\frac{\partial}{\partial n}r(a,n) \ge 1$. 
Thus, $r(a,n)$ increases with $n$ for every fixed $a$. For $ 1 \le n < n_0$,
\begin{equation}\label{rub}
1< -\frac{1}{a} W(-a)=r(a,1) \le r(a,n) < r(a,n_0) = \frac{1}{a} \log \frac{1}{a}.
\end{equation}

\subsection{The case of a real saddle point}

\begin{proposition}\label{propreal}
Let $\eta>0$ be fixed. 
There is a positive constants $a_0>0$ such that for $0<a\le a_0$ and all integers $n$ with
$
1 \le n \le \frac{1}{a}\log \frac{1-\eta}{ea}
$
we have 
$$
g_a(-n) a e^\gamma = \frac{n! \, e^{r-I(ar)}}{r^n \sqrt{2\pi r\left(1-a e^{a r}\right)}} \left(1+O(n^{-1}+ a \log (1/a) )\right),
$$
where $r=r(a,n)$ is given by \eqref{rdef}. Hence $g_a(-n)>0$, provided $n$ is sufficiently large and $a$ is sufficiently small.
\end{proposition}

\begin{corollary}\label{corlambdalb}
Let $\eta>0$. For all sufficiently small $a>0$, the function $g_a(s)$ does not vanish in the half-plane $\re(s)>  -\frac{1}{a}\log  \frac{1-\eta}{e a}$.
\end{corollary}

\begin{proof}[Proof of Corollary \ref{corlambdalb}]
Since $\rho_a(u) \ge 0$, Lemma \ref{FaInv} implies that among the zeros of $g_a(s)$ with maximal real part must be a real zero. 
Assume that a zero with maximal real part is at $s=-\lambda_a$ with $0\le \lambda_a \le  \frac{1}{a}\log  \frac{1-\eta}{e a}$.
Let $n=\lceil \lambda_a \rceil$, so that $n \le \frac{1}{a}\log  \frac{1-\eta/2}{e a}$. The proposition implies that both
$g_a(-n)$ and $g_a(-n+1)$ are positive and stay positive as $a$ decreases to $0$. The zero at $s=-\lambda_a$ must vary continuously with $a$.
From \eqref{eqgaseval} it follows that $g_a(s) a e^\gamma = 1 +O(as)$ when $|as|\le 1$. Thus
 $g_a(s)>0$ for all $s \in [-n, -n+1]$ when $a$ is sufficiently small. It follows that the zero at $s=-\lambda_a$  would have to 
become a complex zero as $a$ decreases. But the right-most zero of $g_a(s)$ cannot be complex, since $\rho_a(u) \ge 0$. 
Thus $g_a(s)$ does not vanish when $\re(s)> - \frac{1}{a}\log  \frac{1-\eta}{e a}$.
\end{proof}

\begin{proof}[Proof of Proposition \ref{propreal}]
First assume $n \ll 1$. Then $r\ll 1$ and \eqref{gapoly} shows that $g_a(-n) a e^\gamma = 1+O(a) >0$, provided $0<a\le a_0$. 
Thus $ g_a(-n) a e^\gamma \ll 1$, which implies the asymptotic estimate when $n\ll 1$. 
In the remainder of this proof we assume that $n$ is sufficiently large and that $n \le \frac{1}{a}\log \frac{1-\eta}{ea}$.
Since $r=n+e^{ar}-1 \ge n$, we may assume that $r$ is large. 

We break up the integral in \eqref{gacint} as 
$$
\int_{-\pi}^\pi  \exp\left(M_{a,r,n}(t)\right)dt = \int_{-\delta}^\delta + \int_{\delta<|t|\le t_0} +\int_{t_0<|t| \le \pi} =: I_1+I_2+I_3,
$$
say, where 
$$
\delta: =\frac{\log r }{\sqrt{r}}
$$
and $t_0$ is a sufficiently small positive constant. 
For $I_2$ and $I_3$ we estimate the absolute value of the integrand. We have
$$
\re(M_{a,r,n}(t)) = -r(1-\cos t ) + \re(I(ar)-I(ar e^{it}))
$$
and
\begin{multline*}
 \re(I(ar)-I(ar e^{it}))  = \int_0^1 (e^{arv}-e^{arv \cos t }\cos(arv \sin t) ) \frac{dv}{v} \\
 =  \int_0^1 e^{arv}(1-\cos(arv \sin t) ) \frac{dv}{v} +  \int_0^1 (e^{arv}-e^{arv \cos t })\cos(arv \sin t) \frac{dv}{v}\\
=Q_1+Q_2,
\end{multline*}
say. We first consider $I_3$, where $|t|>t_0 \gg 1$, $r(1-\cos t) \gg r$. When $art \le 2$ we have $ar\ll 1$ and 
$$
Q_1 \le \int_0^{1} e^{arv}( a r v  t)^2  \frac{dv}{v} \ll 1.
$$ 
When $art > 2$, we write
\begin{multline*}
Q_1 \le \int_0^{1/(ar t)} e^{arv}( a r  t)^2 v dv +  \int_{1/(ar t)}^{1/2} e^{arv} 2 \frac{dv}{v} + \int_{1/2}^1 e^{arv} 2 \frac{dv}{v}\\
\le t e^{1/t} + 2 (ar t) \frac{e^{ar/2}}{ar} + 4 \frac{e^{ar}}{\max(1,ar)}  \ll \frac{e^{ar}}{1+ar}.
\end{multline*}
Similarly,
\begin{multline*}
Q_2 \le \int_0^{1} e^{arv}( 1-e^{-arv t^2/2}) \frac{dv}{v}\le \int_0^{1/2} e^{arv} ar (t^2/2) dv +  \int_{1/2}^1  e^{arv} 2 dv \\
\ll e^{ar/2} + \frac{e^{ar}}{1+ar} \ll \frac{e^{ar}}{1+ar}.
\end{multline*}
Since $e^{ar}=r-n+1 \le r$, it follows that there is a constant $c>0$ such that 
$$
|I_3| \ll \exp( - c r + O( e^{ar}/(1+ar))) \ll \exp( - c r/2 ) \ll 1/r^2.
$$

Next we estimate $I_2$, where $\delta \le |t| \le t_0$. Given any $\varepsilon >0$, we can choose $t_0 $ sufficiently small such that for $|t|\le t_0$ we have
$$
 r(1-\cos t) \ge \frac{r t^2}{2+\varepsilon},
$$
$$
Q_1 \le \int_0^1 e^{arv} \frac{(arv \sin t)^2}{2} \frac{dv}{v} \le \frac{(art)^2}{2} \int_0^1 v e^{arv} dv \le \frac{ar t^2 e^{ar}}{2},
$$
and
$$
Q_2 \le \int_0^{1} e^{arv}( 1-e^{-arv t^2/2})  \frac{dv}{v} \le  \int_0^{1} e^{arv} ar t^2/2 dv \le \frac{e^{ar}t^2}{2}.
$$
Now $r-n = e^{ar}-1 < \frac{1}{a}$, by \eqref{nrdef} and \eqref{rub}. Thus,
\begin{equation}\label{rubsharp}
r = n + (r-n) \le \frac{1}{a}\log \frac{1-\eta}{ae} + \frac{1}{a} = \frac{1}{a}\log \frac{1-\eta}{a},
\end{equation}
i.e. $ar e^{ar} \le r(1-\eta)$. 
It follows that there is a constant $c>0$ such that 
$$
 -r(1-\cos t ) +Q_1+Q_2 \le -\left(\frac{r }{2+\varepsilon} -\frac{e^{ar}(ar+1)}{2}\right) t^2 < - c r \delta^2 = -c \log^2 r
$$
and
$$
|I_2| \ll \exp(-c \log^2 r) \ll 1/r^2. 
$$

For $I_1$ we replace $n$ by $ r-e^{ar}+1$ and approximate $M_{a,r,n}(t)$ by its third degree Taylor polynomial around $t=0$.
We have
\begin{multline*}
\exp(M_{a,r,n}(t)) = \exp\left(-\frac{r}{2}  \left(1-a e^{a r}\right) t^2+ i  \frac{c_{3}}{6} t^3 + O(r  (1+a r)^2  t^4) \right) \\
=\exp(- c_{2} t^2/2) \left\{1+ i  \frac{c_{3}}{6} t^3+O\left(r  (1+a r)^2  t^4 + r^2(1+ar)^2 t^6\right)\right\},
\end{multline*}
since 
$$ c_3 = -r + (1+ar) ar e^{ar} \ll r (1+ar),
$$
by \eqref{rubsharp}, and
$$
\frac{d^4}{dt^4}M_{a,r,n}(t) \ll r (1+ar)^2 \qquad (t \in \mathbb{R}).
$$

Now $r< \frac{1}{a} \log \frac{1}{a}$ implies that $c_2:=r(1-a e^{ar})$ is positive and
$$
\int_{-\infty}^\infty e^{-c_2 t^2/2} t^k dt \ll_k c_2^{-(k+1)/2}.
$$
Thus the contribution from the error term to $I_1$ is 
$$\ll r (1+ar)^2  c_{2}^{-5/2} + r^2 (1+ar)^2 c_{2}^{-7/2}.$$
The contribution from the term $i t^3 c_{3}$ to $I_1$ vanishes by symmetry.
The contribution from the term $1$ to $I_1$ is 
$$
\int_{-\infty}^\infty e^{-\frac{1}{2} t^2 c_{2}}dt +O\left(\int_\delta^\infty e^{-\frac{1}{2} t^2 c_{2}} dt \right)
=\sqrt{\frac{2\pi}{c_2}} + O\left(\frac{e^{-\delta^2 c_2/2}}{\delta c_2}\right).
$$
Thus
$$
I_1 = \sqrt{\frac{2\pi}{c_2}} \left(1+O\left(\frac{e^{-\delta^2 c_2/2}}{\delta \sqrt{c_2}}
+r (1+ar)^2  c_{2}^{-2} + r^2 (1+ar)^2 c_{2}^{-3}\right) \right).
$$
From \eqref{rubsharp} we have $\eta r \le c_2 \le r$, which implies
$$
I_1 =\sqrt{\frac{2\pi}{c_2}} \left(1+O(r^{-1}+r a^2)\right).
$$
Since $n\le r$, $ar < \log \frac{1}{a}$, $c_2=r\left(1-a e^{a r}\right)$ and $|I_2|+|I_3|\ll 1/r^2$, 
the result now follows from \eqref{gacint}.
\end{proof}

\subsection{The case of complex saddle points}

Here we assume 
$
n >  1 + \frac{1}{a} \log \frac{1}{ea},
$
where formula \eqref{rdef} yields a complex value of $r$.  
Define
$$
L:= \log \frac{1}{a},
$$
and
$$
\alpha := \frac{(3 \nu a L)^{2/3}}{2},
$$
where $\nu >0$ is a constant to be determined later. 
We choose
$$
n:= 1 + \left\lfloor \frac{1}{a} \log \frac{1+\alpha}{e a} \right\rfloor = \frac{1}{a} \log \frac{1+\alpha}{e a} + O(1)
$$
and
$$
r:= \frac{1}{a}\left( L + \alpha\left(\frac{1}{3}+\frac{2}{L}\right) \right).
$$
Then
$$
a e^{ar} = e^{\alpha\left(\frac{1}{3}+\frac{2}{L}\right) } = 1+ \alpha\left(\frac{1}{3}+\frac{2}{L}\right) +O(\alpha^2). 
$$
Define
$$
K(\nu):= \int_{-\infty}^\infty \cos\left( -\frac{3\nu}{2}u +\frac{\nu}{2}u^3 \right) du .
$$
Although we will not need this, \cite[Eqs. (9.5.1) and (9.7.9)]{NIST} implies that
\begin{equation}\label{KAiry}
K(\nu)= 2 \pi \left(\frac{2}{3 \nu}\right)^{1/3} \mathrm{Ai}\left( - \left(\frac{3 \nu}{2}\right)^{2/3} \right) = \sqrt{\frac{8\pi}{3 \nu}}\left(\cos(\nu-\pi/4)+O(\nu^{-1})\right),
\end{equation}
where $\mathrm{Ai}(z)$ is the Airy function.  

 \begin{figure}[h]
\begin{center}
\includegraphics[height=5.4425cm,width=10cm]{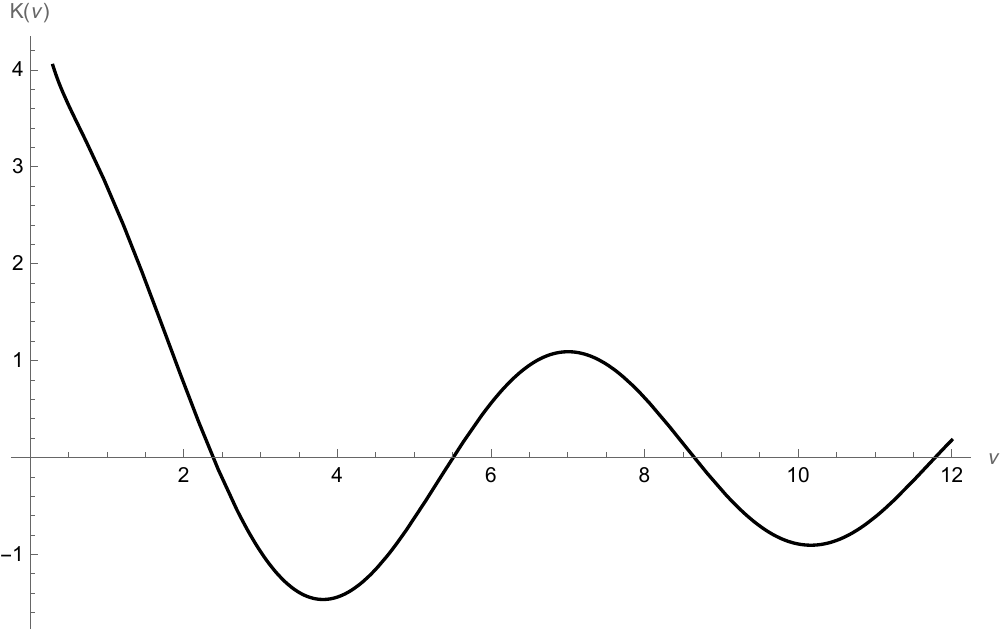}
\caption{The graph of $K(\nu)$.}
\label{figure2}
\end{center}
\end{figure}

\begin{proposition}\label{proplambdaub}
Let $\nu>0$ be fixed. With the notation as above, we have 
\begin{equation*}
 g_a(-n) a e^\gamma 
= \frac{n!}{2\pi } \frac{e^{r-I(ar)}}{r^n}  \frac{\sqrt{2\alpha}}{L} ( K(\nu) + o(1)),
\end{equation*}
as $a \to 0$. 
\end{proposition}

\begin{corollary}\label{corlambdaub}
Let $\nu_i$,  $i\ge 0$, denote the positive zeros of $K(\nu)$ in increasing order. The first four zeros (see Figure \ref{figure2}) are
$$
\nu_0 = 2.383446..., \quad \nu_1 = 5.510195..., \quad \nu_2 = 8.647357..., \quad \nu_3 = 11.786842... 
$$
The proposition implies that for each fixed integer $k\ge 0$, $g_a(s)$ has a zero at
$$
s_k = -\frac{1}{a} \log \frac{1+ \frac{1}{2}(3 (\nu_k+o(1)) a L)^{2/3}}{e a} 
=  -\frac{1}{a}\left( L-1 + \frac{1+o(1)}{2}(3 \nu_k a L)^{2/3}\right)
$$
as $a\to 0$. In particular, this shows that
$$
\lambda_{a} \le  -s_0 = \frac{1}{a}\left( L-1 + \frac{1+o(1)}{2}(3 \nu_0 a L)^{2/3}\right) 
= 
\frac{1}{a} \left(\log\left(\frac{1}{a}\right)-1+o(1)\right),
$$
as $a\to 0$.
\end{corollary}

From \eqref{KAiry} it follows that $\nu_k = \pi \left(k + \frac{3}{4}\right) +o(1)$, as  $k\to \infty$, but only the first zero at $\nu_0$ is needed for the proof of Theorem \ref{thmlambda}.

\begin{proof}[Proof of Proposition \ref{proplambdaub}]
We approximate the exponent of the integrand in \eqref{gacint} by its fourth degree Taylor polynomial around $t=0$.
Here the linear term does not vanish.
We have
\begin{multline*}
M_{a,r,n}(t) = i  \left(r+1-n-e^{ar}\right)t
+\frac{r}{2}  \left(a e^{a r}-1\right) t^2
+ \frac{i r}{6} \left(a  e^{a r} (a r+1)-1\right) t^3 \\
+ \frac{r}{24}\left(1-a  e^{a r} \left(a^2 r^2+3 a r+1\right)\right) t^4 
+O(e^{ar} (ar)^4 |t|^5) \\
=: ic_1 t+ c_{2} t^2+ i  c_{3} t^3+  c_{4} t^4 +O( a^{-1} L^4 |t|^5),
\end{multline*}
since $\frac{d^5}{dt^5} M_{a,r,n}(t) \ll e^{ar} (ar)^4 \ll a^{-1} L^4$ for all $t\in \mathbb{R}$.
The coefficients are
$$
c_1 = r+1-n-e^{ar}= -\frac{\alpha}{a} +O(1),
$$
$$
c_2= \frac{r}{2}  \left(a e^{a r}-1\right)
=\frac{\alpha L}{6a}\left(1+O\left(\frac{1}{L}\right)\right),
$$
$$
c_3= \frac{r}{6} \left(a  e^{a r} (a r+1)-1\right) = \frac{L^2}{6a} ( 1+ O(\alpha)),
$$
and
$$
c_4 =  \frac{r}{24}\left(1-a  e^{a r} \left(a^2 r^2+3 a r+1\right)\right) = -\frac{L^3}{24 a} \left(1+O\left(\frac{1}{L}\right)\right).
$$

Define
$$
T=\frac{\alpha^{1/3}}{L^2}, \qquad T_1 = \frac{\alpha^{2/5}}{L^2}.
$$
For $|t|\le T$, we have $a^{-1}  L^4 |t|^5 \ll 1$. Thus,
\begin{equation*}
\begin{split}
 I_1 & := \int_{-T}^T \exp\{ M_{a,r,n}(t) \} dt   \\
& =\int_{-T}^T \exp\{c_2 t^2 + c_4 t^4\} \exp\{i(c_1 t + c_3 t^3)\} (1+O(L^4 a^{-1} T^5)) dt \\
& = O(L^4  a^{-1}  T^6)+ \int_{-T}^T \exp\{c_2 t^2 + c_4 t^4\}\exp\{i(c_1 t + c_3 t^3)\} dt, \\
& = O(\sqrt{\alpha}/L^7)+ \int_{-T}^T \exp\{c_2 t^2 + c_4 t^4\} \cos(c_1 t + c_3 t^3) dt, \\
\end{split}
\end{equation*}
since $ \exp\{c_2 t^2 + c_4 t^4\} \ll 1$ for all $t\in \mathbb{R}$.
Indeed, $c_2 t^2 + c_4 t^4$ has its global maximum at $t=\pm T_0$, say, where $T_0 \sim  \sqrt{2\alpha}/L$ and 
$c_2 T_0^2 +c_4 T_0^4 \ll 1$. 
Since $ \exp\{c_2 t^2 + c_4 t^4\}$ is decreasing on $[T_0, \infty]$ and $T_1>T_0$, the second mean value theorem for integrals 
implies that there exists a $T_2 \in [T_1,T]$ such that
$$
\int_{T_1}^T \exp\{c_2 t^2 + c_4 t^4\} \cos(c_1 t + c_3 t^3) dt
= \exp\{c_2 T_1^2 + c_4 T_1^4\} \int_{T_1}^{T_2} \cos(c_1 t + c_3 t^3) dt.
$$
A little exercise shows that the last integral is $\ll \frac{1}{c_3 T_1^2}$, so that 
$$
\int_{T_1}^T \exp\{c_2 t^2 + c_4 t^4\} \cos(c_1 t + c_3 t^3) dt \ll \frac{ 1}{c_3 T_1^2}\asymp \alpha^{7/10}L .
$$
When $|t| \le T_1$, we have
$$
|c_2 t^2 + c_4 t^4| \le |c_2 T_1^2 + c_4 T_1^4|  \asymp |c_4 T_1^4| \ll 1,
$$
which yields
$$
\exp\{c_2 t^2 + c_4 t^4\} = 1+ O( |c_4| T_1^4) 
$$
and
$$
 \int_{-T_1}^{T_1} \exp\{c_2 t^2 + c_4 t^4\} \cos(c_1 t + c_3 t^3) dt = O(|c_4| T_1^5)+ \int_{-T_1}^{T_1} \cos(c_1 t + c_3 t^3) dt .
$$
Note that $|c_4| T_1^5 \asymp \sqrt{\alpha} L^{-6}$. 
After the change of variable $u=t L/\sqrt{2 \alpha}$ and $U=T_1 L/\sqrt{2 \alpha} \asymp \alpha^{-1/10} /L$, 
 we have
\begin{equation*}
\begin{split}
&\frac{L}{\sqrt{2\alpha}}\int_{-T_1}^{T_1} \cos(c_1 t + c_3 t^3) dt\\
& =  \int_{-U}^U \left[\cos\left( -\frac{3\nu}{2}u +\frac{\nu}{2}u^3 \right) +O(U\sqrt{\alpha}/L+U^3 \alpha)\right] du\\
 & = O(  U^2 \sqrt{\alpha}/L + U^4 \alpha)+  \int_{-U}^U \cos\left( -\frac{3\nu}{2}u +\frac{\nu}{2}u^3 \right) du\\
 & = O(U^2 \sqrt{\alpha}/L  + U^4 \alpha +U^{-2})+  \int_{-\infty}^\infty \cos\left( -\frac{3\nu}{2}u +\frac{\nu}{2}u^3 \right) du\\
 &  = o(1)+  K(\nu).
\end{split}
\end{equation*}
Combining everything, we obtain
$$
  \int_{-T}^T \exp\{ M_{a,r,n}(t) \} dt = \frac{\sqrt{2\alpha}}{L} ( K(\nu) + o(1)),
$$
as $a \to 0$. 

It remains to show that 
$$
 \int_{T}^{\pi} \exp\{ M_{a,r,n}(t) \} dt = o\left(\frac{\sqrt{2\alpha}}{L} \right).
$$
Let $t_0 = 7/L$ and write
$$
 \int_{T}^{\pi} \exp\{ M_{a,r,n}(t) \} dt = \int_{T}^{t_0} + \int_{t_0}^{\pi} =: I_2 +I_3.
$$
For $I_2$ and $I_3$, we estimate the absolute value of the integrand. As in the case of a real saddle point,  we have
$$
\re(M_{a,r,n}(t)) = -r(1-\cos t ) + Q_1+Q_2,
$$
where
\begin{equation*}
Q_1+Q_2 =  \int_0^1 e^{arv}(1-\cos(arv \sin t) ) \frac{dv}{v} +  \int_0^1 ((e^{arv}-e^{arv \cos t })\cos(arv \sin t) ) \frac{dv}{v}.
\end{equation*}
We note that 
$$\frac{t^2}{5} \le 1-\cos t \le  \frac{t^2}{2}\qquad (|t| \le \pi).
$$ 
To estimate $I_3$, we write
\begin{multline*}
Q_1 \le \frac{1}{2}\int_0^{1/(ar t)} e^{arv}( a r  t)^2 v dv +  \int_{1/(ar t)}^{1/2} e^{arv} 2 \frac{dv}{v} + \int_{1/2}^1 e^{arv} 2 \frac{dv}{v}\\
\le \frac{t}{2}e^{1/t} + 2 (ar t) \frac{e^{ar/2}}{ar} + 4 \frac{e^{ar}}{ar}  \le  \frac{5r}{L^2},
\end{multline*}
for $a$ sufficiently small and $|t|\ge t_0$.
Since $1-e^{-x} \le x$ for all $x$, 
\begin{multline*}
Q_2 \le \int_0^{1} e^{arv}( 1-e^{-arv t^2/2}) \frac{dv}{v}\le \int_0^{1/2} e^{arv} ar (t^2/2) dv +  \int_{1/2}^1  e^{arv} 2 dv \\
\le e^{ar/2}\pi^2/2 + \frac{2 e^{ar}}{ar} \le \frac{3r}{L^2}.
\end{multline*}
We have $r(1-\cos t) \ge r t^2/5 \ge r t_0^2/5$. Combining these three estimates we obtain
$$
|I_3| \le 2\pi \exp( - rt_0^2/5 +8r/L^2) \le 2\pi \exp(-r/L^2)= o\left(\frac{\sqrt{2\alpha}}{L} \right).
$$

Next, we estimate $I_2$. For all real $t$ we have
$$
 r(1-\cos t) \ge r \left(\frac{ t^2}{2}-\frac{t^4}{24}\right).
$$

When $|t| \le t_0 = 7/L$, then $|\sin t| \le 7/L$ and 
$$
1-\cos(arv \sin t) \le \frac{(arv t)^2}{2}- \frac{(arvt)^4}{\kappa},
$$
for some suitable constant $\kappa$. Thus 
$$
Q_1 \le \int_0^1 e^{arv} \left(\frac{(arv  t)^2}{2} -\frac{(arvt)^4}{\kappa}\right) \frac{dv}{v}
\le  \frac{t^2}{2} \left( ar e^{ar} - e^{ar}+1\right)- \frac{t^4}{2\kappa} (ar)^3 e^{ar}
$$
and
$$
Q_2 \le \int_0^{1} e^{arv}( 1-e^{-arv t^2/2})  \frac{dv}{v} \le  \int_0^{1} e^{arv} \frac{ar t^2}{2} dv = \frac{t^2}{2} (e^{ar}-1).
$$
For sufficiently small values of $a$ and  $t_0\ge |t| \ge T = \frac{\alpha^{1/3}}{L^2}$, we obtain
\begin{equation*}
\begin{split}
  -r(1-\cos t) + Q_1+Q_2 
 \le & -r\left(\frac{ t^2}{2}-\frac{t^4}{24}\right)+ \frac{t^2}{2} ar e^{ar}  - \frac{t^4}{2\kappa} (ar)^3 e^{ar}\\
\le & \frac{t^2}{2} r \left(-1 + a e^{ar} \right) - \frac{t^4}{3\kappa} (ar)^3 e^{ar}\\
 \le & \frac{t^2}{2} r \alpha - \frac{t^4}{4\kappa} r L^2 
\le  - \frac{t^4}{5\kappa} r   L^2 
\le  -\frac{T^4}{5\kappa} r   L^2 \\
= &  -\frac{\alpha^{4/3}}{5\kappa L^6} r \le - a^{-1/10}.
\end{split}
\end{equation*}
It follows that 
$
|I_2| \le 2 \pi \exp(- a^{-1/10}) \ll a  = o\left(\frac{\sqrt{2\alpha}}{L} \right).
$
Combining everything, we have 
$$
  \int_{-\pi}^\pi \exp\{ M_{a,r,n}(t) \} dt = \frac{\sqrt{2\alpha}}{L} ( K(\nu) + o(1)),
$$
as $a \to 0$. 
The result now follows from \eqref{gacint}.
\end{proof}

\end{document}